\documentclass[a4paper,10pt,oneside]{amsart}

\usepackage{amsmath,amssymb,amsthm}
\usepackage{mathrsfs}
\usepackage{url,graphicx}
\usepackage{xcolor}
\definecolor{aliceblue}{rgb}{0.94, 0.97, 1.0}
\usepackage[all]{xy}
\usepackage{nicefrac}
\usepackage{array}
\usepackage{tabularx}
\usepackage{hyperref}
\hypersetup{
    colorlinks,
    linkcolor={red!50!black},
    citecolor={blue!50!black},
    urlcolor={blue!80!black}
}

\theoremstyle{plain}
\newtheorem{lemma}{Lemma}[section]
\newtheorem{proposition}[lemma]{Proposition}
\newtheorem{theorem}[lemma]{Theorem}
\newtheorem*{theorem-nonum}{Theorem}

\theoremstyle{remark}
\newtheorem{remark}[lemma]{Remark}

\theoremstyle{definition}
\newtheorem{definition}[lemma]{Definition}

\newcommand{\p}{\mathbb{P}}
\newcommand{\C}{\mathbb{C}}

\newcommand{\Z}{\mathbb{Z}}
\newcommand{\N}{\mathbb{N}}
\newcommand{\T}{\mathbb{T}}

\newcommand{\pgl}{\p\!\operatorname{GL}}
\newcommand{\ord}{\operatorname{ord}}
\newcommand{\bl}[2]{\operatorname{Bl_{\,{#1}}}\!\left(#2\right)}
\newcommand{\Gr}[2]{\operatorname{\mathbb{G}}\!\left(#1, #2\right)}
\newcommand{\chern}[2]{\operatorname{c}_{#1}\!\left(#2\right)}
\newcommand{\chernT}[2]{\operatorname{c}_{#1}^{\T}\!\left(#2\right)}

\renewcommand{\tilde}{\widetilde}
\newcommand{\tth}{\thinspace}
\newcommand{\mscr}[1]{\mathscr{#1}}
\newcommand{\mcal}[1]{\mathcal{#1}}

\title{Counting projections of rational curves}

\author[M. Gallet]{
Matteo Gallet$^{\ast, \circ}$}
\address[Matteo Gallet]{International School for Advanced Studies/Scuola 
Internazionale Superiore di Studi Avanzati (ISAS/SISSA),
Via Bonomea 265, 34136 Trieste, Italy}
\email{mgallet@sissa.it}

\author[J. Schicho]{
Josef Schicho$^{\ast}$}
\address[Josef Schicho]{Research Institute for Symbolic Computation (RISC), 
Johannes Kepler University, Linz}
\email{jschicho@risc.jku.at}

\thanks{This version of the article has been accepted for publication, 
after peer review (when applicable) and is subject to Springer Nature's \href{https://www.springernature.com/gp/open-research/policies/accepted-manuscript-terms}{AM terms of use}, 
but is not the Version of Record and does not reflect post-acceptance improvements, or any corrections. 
The Version of Record is available online at: \url{https://doi.org/10.1007/s11856-020-2071-3}.}

\thanks{$^\ast$ Supported by the Austrian Science Fund (FWF): W1214-N15 
Project DK9 and P26607.} 

\thanks{$^\circ$ Supported by the Austrian Science Fund (FWF): P25652 and Erwin 
Schr\"odinger Fellowship J4253.}

\begin{document}

\begin{abstract}
 Given two general rational curves of the same degree in two projective spaces, 
one can ask whether there exists a third rational curve of the same degree that 
projects to both of them. We show that, under suitable assumptions on the degree 
of the curves and the dimensions of the two given ambient projective spaces, the 
number of curves and projections fulfilling the requirements is finite. Using 
standard techniques in intersection theory and the Bott residue formula, we 
compute this number.
\end{abstract}

\maketitle

\section*{Introduction}

Inspired by problems in multiview geometry concerning image-object 
correspondence under projections (see for example~\cite{Burdis2013}, 
and~\cite{Hartley2004} for a general account on the topic), we consider the 
following question, where all varieties are complex:
\begin{itemize}
 \item[\textbf{given}] two general rational curves $C_a 
\subseteq \p^a$ and $C_b  \subseteq \p^b$, both of degree $d \in \N$, 
and a natural number $c \in \N$,
 \item[\textbf{find}] a rational curve $C_c \subseteq  \p^c$ of 
degree~$d$, together with two linear projections
\[
 \pi_a \colon \p^c \dashrightarrow \p^a, \quad \pi_b \colon \p^c 
\dashrightarrow \p^b
\]
such that $\pi_a(C_c) = C_a$ and $\pi_b(C_c) = C_b$.
\end{itemize}
We are interested in counting the number of such curves and projections, when 
this number is finite up to projective equivalence in $\p^c$.

Because of the rationality assumption, we can interpret the curves appearing in 
the previous formulation as images of maps $f_u \colon \p^1 \longrightarrow 
\p^u$ for $u \in \{a,b,c\}$. In this way, we can translate the original 
problem into a problem of vector spaces of polynomials of degree~$d$ 
on~$\p^1$. Since we are only interested in the images of the maps~$f_u$, 
and not in the maps themselves, we need to allow possible reparametrizations, 
namely automorphisms of~$\p^1$. Once we apply this translation, the problem 
becomes:
\begin{itemize}
 \item[\textbf{given}] two general vector subspaces $V_a, V_b \subseteq 
\C[s,t]_d$ of dimension $a+1$ and $b+1$, respectively, and a natural number $c 
\in \N$,
 \item[\textbf{find}] automorphisms $\sigma \in \pgl(2,\C)$ such that
\[
 \dim \bigl( V_a + V_b^{\sigma} \bigr) \, \leq \, c+1.
\]
\end{itemize}
Here $\C[s,t]_d$ is the vector space of homogeneous polynomials of degree~$d$, 
and we denote by~$V_b^{\sigma}$ the image of~$V_b$ under the action 
of~$\sigma$, which operates on polynomials by applying the change of 
coordinates determined by~$\sigma$ to the variables.

A dimension count shows that one may expect that if $V_a$ and $V_b$ are general 
subspaces, and the condition
\begin{equation}
\label{equation:finiteness_condition}
\tag{$\ast$}
 (a + b - c + 1)(d - c) \, = \, 3
\end{equation}
holds, then there exists a finite number of automorphisms $\sigma \in 
\pgl(2,\C)$ satisfying the requirements of the problem. In terms of the initial 
formulation, this means that if Equation~\eqref{equation:finiteness_condition} 
holds and the curves $C_a$ and $C_b$ are general, then one may expect that 
there exist finitely many --- up to changes of coordinates in~$\p^c$ --- 
curves $C_c$ and projections~$\pi_a$ and~$\pi_b$ sending~$C_c$ to~$C_a$ 
and to~$C_b$, respectively.

\medskip
The aim of this paper is, under the assumption that 
Equation~\eqref{equation:finiteness_condition} 
holds and that the vector subspaces~$V_a$ and~$V_b$ are general,
\begin{itemize}
 \item[-] 
  to prove that the number of automorphisms $\sigma \in \pgl(2,\C)$ satisfying 
the requirements of the problem is indeed finite;
 \item[-]
  to provide a formula for this number in terms of the parameters $a,b,c,d$. 
\end{itemize}

This is the main result of our work (see 
Theorem~\ref{theorem:number_witnesses}): 
\begin{theorem-nonum}
  Let $C_a \subseteq \p^a$ and $C_b \subseteq \p^b$ be 
two general rational curves of degree~$d$. Let $c$ be a natural number and 
suppose that Equation~\eqref{equation:finiteness_condition} holds. Then there 
are, up to automorphisms of~$\p^c$, finitely many rational curves $C_c 
\subseteq \p^c$ of degree~$d$ together with linear projections $\pi_a 
\colon C_c \longrightarrow C_a$ and $\pi_b \colon C_c \longrightarrow C_b$.
\begin{itemize}
  \item[(1)]
    Suppose that $a + b + 1 - c = 1$ and $d - c = 3$. Then, the number of these 
curves and projections is
    \[
      \frac{1}{6} (a+3)(a+2)(a+1)(b+3)(b+2)(b+1).
    \]
  \item[(2)]
    Suppose that $a + b + 1 - c = 3$ and $d - c = 1$. Then, the number of these 
curves and projections is
    \[ 
      \frac{1}{6} \tth ab (a^2 - 1)(b^2 - 1).
    \]
\end{itemize}
\end{theorem-nonum}

The paper is structured as follows. In Section~\ref{linearization} we operate 
the translation from the first to the second formulation of the problem. In 
Section~\ref{finiteness} we prove that the number of solutions to our problem 
is finite when Equation~\eqref{equation:finiteness_condition} holds. 
Eventually, in Section~\ref{formula} we prove the formulas counting the number 
of solutions by means of intersection theory and the Bott residue formula.

\section{From curves to linear systems on~\texorpdfstring{$\p^1$}{the 
projective line}}
\label{linearization}

As explained in the Introduction, we are given two general rational curves $C_a 
\subseteq \p^a$ and $C_b \subseteq \p^b$, both of 
degree~$d$, and we ask whether there exists a non-degenerate rational curve 
$C_c \subseteq \p^c$, together with linear projections to both~$C_a$ 
and~$C_b$. 
The rationality of the curves allows us to use their parametrizations in order 
to attack this problem. This comes at a cost: since we are only interested in 
the curves, we need to take into account the possibility of reparametrizations. 

If for $u \in \{a,b,c\}$ the morphism $f_{u} \colon \p^1 \longrightarrow 
\p^{u}$ is a parametrization of the curve~$C_u$, then our problem translates to:

\begin{itemize}
 \item[\textbf{given}]
  two general morphisms $f_a \colon \p^1 \longrightarrow \p^a$ and $f_b
\colon \p^1 \longrightarrow \p^b$ of degree $d$, and a number $c \in \N$,
 \item[\textbf{find}]
  a morphism $f_c  \colon \p^1 \longrightarrow \p^c$, 
together with two linear projections $\pi_a$, $\pi_b$ and two isomorphisms 
$\sigma_a, \sigma_b \colon \p^1 \longrightarrow \p^1$ making the following 
diagram commutative:
  \begin{equation}
  \begin{gathered}
  \label{equation:diagram}
   \xymatrix{
     & \p^1 \ar[rr]^{f_a} & & \p^a \\
     \p^1 \ar[rr]^{f_c} \ar[ru]^{\sigma_a} \ar[rd]_{\sigma_b} & & 
      \p^c \ar@{-->}[ru]_{\pi_a} \ar@{-->}[rd]^{\pi_b} \\
     & \p^1 \ar[rr]^{f_b} & & \p^b
   }
  \end{gathered}
  \end{equation}
\end{itemize}

We notice that, by eventually re-defining the map~$f_c$, we can always suppose 
that in Diagram~\eqref{equation:diagram} the map~$\sigma_a$ is the identity, 
so we ask whether there exists an automorphism $\sigma \in \pgl(2, \C)$ such 
that the following diagram commutes.
\begin{equation}
\label{equation:diagram_simplified}
\begin{gathered}
 \xymatrix{
  & & \p^a \\
  \p^1 \ar[r]^{f_c} \ar@/^1pc/[rru]^{f_a} \ar@/_1pc/[rrd]_{f_b \circ \sigma} & 
  \p^c \ar@{-->}[ru]_{\pi_a} \ar@{-->}[rd]^{\pi_b} \\
  & & \p^b
 }
\end{gathered}
\end{equation}
Let us write
\[
 f_a = (F_0 : \dotsc : F_a), \quad 
 f_b = (G_0 : \dotsc : G_b), \quad 
 f_c = (H_0 : \dotsc : H_c),
\]
where $\{F_i\}_{i=0}^{a}$, $\{G_j\}_{j=0}^{b}$ and $\{H_k\}_{k=0}^{c}$ are 
homogeneous polynomials of degree~$d$ in~$\C[s,t]$. If we denote by 
$G_j^{\sigma}$ the image of~$G_j$ under the action of~$\sigma$, then the 
commutativity of Diagram~\eqref{equation:diagram_simplified} is equivalent to
\begin{equation}
\label{equation:explicit_dependence}
\begin{gathered}
 (F_0 : \dotsc : F_a) = 
 \left( 
  \sum \lambda_{0j} H_j : \dotsc : \sum \lambda_{aj} H_j 
 \right) \\
 (G_0^{\sigma} : \dotsc : G_b^{\sigma}) =  
 \left( 
  \sum \mu_{0j} H_j : \dotsc : \sum \mu_{bj} H_j 
 \right)
\end{gathered}
\end{equation}
for some complex coefficients $\{ \lambda_{ij} \}_{i \in \{0, 
\dotsc, a\}}^{j \in \{0, \dotsc, c\}}$ and $\{ \mu_{ij} \}_{i \in \{0, 
\dotsc, b\}}^{j \in \{0, \dotsc, c\}}$. If we define
\[
 V_a := \left\langle F_0, \dotsc, F_a \right\rangle
 \quad \text{and} \quad
 V_b := \left\langle G_0, \dotsc, G_b \right\rangle, 
\]
then, taking into account Equation~\eqref{equation:explicit_dependence}, one 
sees that the existence of maps~$f_c$, projections~$\pi_a$ and~$\pi_b$ 
and automorphisms~$\sigma$ as in Diagram~\eqref{equation:diagram_simplified} is 
equivalent to the existence of automorphisms~$\sigma$ such that 
\[
 \dim \bigl( V_a + V_b^{\sigma} \bigr) \, \leq \, c+1,
\]
where recall that we denote by~$V_b^{\sigma}$ the image of~$V_b$ under the 
action of~$\sigma$, which operates by a change of coordinates. 

This completes the translation of our initial problem into a question 
concerning vector subspaces of the space of binary homogeneous polynomials of 
degree~$d$.
Therefore from now on we will be concerned with the problem (already reported 
in the Introduction):
\begin{itemize}
 \item[\textbf{given}] two vector subspaces $V_a, V_b \subseteq \C[s,t]_d$ of 
dimension~$a+1$ and~$b+1$, respectively, and a natural number $c \in \N$,
 \item[\textbf{find}] automorphisms $\sigma \in \pgl(2,\C)$ such that
\[
 \dim \bigl( V_a + V_b^{\sigma} \bigr) \, \leq \, c+1.
\]
\end{itemize}

We conclude this section by mentioning a duality between instances and 
solutions of the previous problem that will be useful in the next 
sections. Consider the following perfect pairing between spaces of binary 
polynomials of degree~$d$:
\begin{equation}
\label{equation:pairing}
  \begin{array}{rccc}
    ( \, \cdot, \cdot \, ) & \C[s,t]_{d} \times \C[s,t]_{d} & 
\longrightarrow & \C \\
    & (p_0 s^d + \dotsb + p_d t^d, \tth q_0 s^d + \dotsb + q_d t^d) & \mapsto & 
\displaystyle \sum_{i=0}^{d} p_i \tth q_{d-i} \frac{(-1)^i}{\binom{d}{i}}
  \end{array}
\end{equation}
This is the unique (up to scaling) pairing that is invariant under the 
action of the group~$\pgl(2,\C)$ by change of coordinates. It turns out to be 
the $d$-th symmetric power of the pairing on the set of linear forms given by 
the determinant (see \cite[Section 1.5.1]{Dolgachev2012}). Using this pairing, 
we can define the orthogonal space to every vector subspace $V \subseteq 
\C[s,t]_d$ by setting
\[
 V^{\perp} := \bigl\{ F \in \C[s,t]_d \, : \, (F,G) = 0 \text{ for all } 
G \in V \bigr\}.
\]
Every automorphism $\sigma \in \pgl(2,\C)$ admits an 
adjoint~$\sigma^{\perp}$ with respect to this pairing, such that $(F^{\sigma}, 
G) = (F, G^{\sigma^{\perp}})$. With these definitions, one can check that there 
is a bijection, given by $\sigma \leftrightarrow \sigma^{\perp}$, between the 
sets:
\begin{equation}
\label{equation:bijection}
 \resizebox{.93\textwidth}{!}{$
 \left\{
 \begin{array}{c}
  \text{automorphisms } \sigma \text{ such that} \\
  \dim \bigl( V_a + V_b^{\sigma} \bigr) \, \leq \, c+1
 \end{array}
 \right\}
 \, \longleftrightarrow \,
 \left\{
 \begin{array}{c}
  \text{automorphisms } \sigma^{\perp} \text{ such that} \\
  \dim \bigl( V_a^{\perp} + (V_b^{\perp})^{\sigma^{\perp}} \bigr) \, 
  \leq \, d-c-1
 \end{array}
 \right\}
 $}
\end{equation}

\section{Finiteness of solutions}
\label{finiteness}

In this section we prove that if $V_a$ and $V_b$ are 
general and Equation~\eqref{equation:finiteness_condition} 
holds, namely
\[
 (a + b - c + 1)(d - c) \, = \, 3,
\]
then the number of automorphisms $\sigma \in 
\pgl(2,\C)$ such that $\dim ( V_a + V_b^{\sigma} ) \leq c+1$ is finite.
To understand why we expect this result, we can argue as follows. 
Since $V_a$ has dimension~$a+1$, then its 
orthogonal~$V_a^{\perp}$ has dimension $d-a$. If we fix a basis $L_1, \dotsc, 
L_{d-a}$ of~$V_a^{\perp}$ and a basis $G_0, \dotsc, G_b$ of $V_b$, then the 
condition $\dim \bigl( V_a + V_b^{\sigma} \bigr) \leq c+1$ is equivalent to 
imposing that the rank of the matrix with entries $(L_i, G_j^{\sigma})$ 
is~$c-a$, where $( \, \cdot, \cdot \, )$ is the pairing defined in 
Section~\ref{linearization}. The matrix is of size $(d-a) \times (b+1)$, and so 
one can expect that this rank condition is a condition on~$\pgl(2,\C)$ of 
codimension (see \cite[Proposition~12.2]{Harris1995})
\[
 \bigl((d-a)-(c-a)\bigr) \bigl((b+1)-(c-a)\bigr) = (d-c)(a+b-c+1).
\]
Since $\pgl(2,\C)$ is three-dimensional, if we want a finite number of solution 
we should suppose that the previous quantity equals~$3$.

Define the incidence variety
\begin{multline*}
 \mscr{I} \, = \, 
 \Bigl\{ 
  \bigl( \sigma, [V_a], [V_b] \bigr) \in \pgl(2,\C) 
  \times \Gr{a}{d} \times \Gr{b}{d} \, : \\[-6pt]
  \dim ( V_a + V_b^{\sigma} ) \leq c+1 
 \Bigr\},
\end{multline*}
where $\Gr{b}{n}$ is the Grassmannian of $b$-dimensional linear 
subvarieties of~$\p^n$, and $[\,\cdot\,]$ denotes the element in the 
Grassmannian determined by a vector subspace of~$\C[s,t]_d$. If we consider the 
projection
\[
 \psi \colon \mscr{I} \longrightarrow \pgl(2,\C) \times \Gr{b}{d}
\]
on the first and the third component, then we see that the fibers of~$\psi$ are 
isomorphic to Schubert varieties of~$\Gr{a}{d}$. In fact, following 
\cite[Section~4.1]{Eisenbud2016}, if we fix $\sigma \in \pgl(2,\C)$ and $[V_b] 
\in \Gr{b}{d}$, we can define a complete flag
\[
 \mathbb{F}: \quad 
 \{0\} \, \subsetneq \, \mathbb{F}_1 \, \subsetneq \, 
 \mathbb{F}_2 \, \subsetneq \, \dotsb \, \subsetneq \, 
 \mathbb{F}_{d+1} = \C[s,t]_d
\]
for~$\C[s,t]_d$ such that $\mathbb{F}_{b+1} = V_b^{\sigma}$. The fiber of~$\psi$ 
over $(\sigma, [V_b])$ can be written as
\[
 \psi^{-1}(\sigma, [V_b]) = 
 \bigl\{ 
  \bigl( \sigma, [V_a], [V_b] \bigr) \, : \, 
  \dim ( V_a \cap V_b^{\sigma} ) \geq a+b-c+1 
 \bigr\}.
\]
One can check that the latter is isomorphic to the Schubert variety
\[
 \Sigma_{\lambda}(\mathbb{F}) = 
 \bigl\{ 
  \Lambda \in \Gr{a}{d} \, : \, 
  \dim ( \mathbb{F}_{d-a+i-\lambda_i} \cap \Lambda) \geq i \;\; 
  \text{for all } i \in \{1, \dotsc, a+1\} 
 \bigr\},
\]
where
\[
 \lambda = (\underbrace{d-c, \dotsc, d-c}_{(a+b-c+1) \text{ times}}, 
 \underbrace{0, \dotsc, 0}_{(c-b) \text{ times}}).
\]
The fibers of~$\psi$ are hence irreducible; moreover, their codimension in 
$\Gr{a}{d}$ is
\begin{equation}
\label{equation:codimension}
 (d-c)(a+b-c+1).
\end{equation}
This implies that also~$\mscr{I}$ is irreducible, and a direct computation 
shows that if Equation~\eqref{equation:finiteness_condition} holds, then
\[
 \dim(\mscr{I}) \, = \, \dim \bigl( \Gr{a}{d} \times \Gr{b}{d} \bigr).
\]
In order to prove our initial claim, it is enough to show that the projection
\[
 \phi \colon \mscr{I} \longrightarrow \Gr{a}{d} \times \Gr{b}{d}
\]
is dominant. By the properties of the dimension of the fiber of a regular map 
(see \cite[Section~6.3, Theorem~1.25]{Shafarevich2013}), 
it follows that in order to show that $\phi$ is dominant, it suffices to 
exhibit a single point in~$\Gr{a}{d} \times \Gr{b}{d}$ whose preimage 
under~$\phi$ is zero-dimensional.

Notice that Equation~\eqref{equation:finiteness_condition} holds if and 
only if 
\begin{center}
\begin{minipage}{0.4\textwidth}
 \begin{equation}
 \label{equation:finiteness_condition_caseA}
 \tag{$\ast.a$}
  \left\{
   \begin{array}{rcl}
    d - c & = & 3 \\
    a + b - c + 1 & = & 1
   \end{array}
  \right.
 \end{equation} 
\end{minipage}%
\begin{minipage}{0.1\textwidth}
 \begin{center}
 \text{or}
 \end{center}
\end{minipage}
\begin{minipage}{0.4\textwidth}
 \begin{equation}
 \label{equation:finiteness_condition_caseB}
 \tag{$\ast.b$} 
  \left\{
   \begin{array}{rcl}
    d - c & = & 1 \\
    a + b - c + 1 & = & 3
   \end{array}
  \right.
 \end{equation}
\end{minipage}
\end{center}
From now on we suppose that we are in the case prescribed by 
Equation~\eqref{equation:finiteness_condition_caseA}, and at the end of 
the section we explain how to handle the situation determined by 
Equation~\eqref{equation:finiteness_condition_caseB}.

We construct a point in~$\Gr{a}{d} \times \Gr{b}{d}$ as follows. Define
\[
 V_a := (g_a)_d \qquad V_b := (g_b)_d,
\]
where $(\cdot)_d$ denotes the $d$-th homogeneous component of a homogeneous 
ideal, and the polynomials~$g_a$ and~$g_b$ are any two homogeneous 
polynomials such that:
\begin{itemize}
 \item[-]
  $\deg(g_a) = d-a$ and $\deg(g_b) = d-b$;
 \item[-]
  both $g_a$ and $g_b$ are squarefree;
 \item[-]
  the cross-ratios of any four roots of~$g_a$ or $g_b$ are different.
\end{itemize}
By construction, we have that $\dim(V_a) = a+1$ and $\dim(V_b) = b+1$, so 
$([V_a],[V_b])$ is a point in~$\Gr{a}{d} \times \Gr{b}{d}$. Moreover, by the 
hypothesis on the cross-ratios, we see that for any fixed $\sigma \in 
\pgl(2,\C)$, the polynomials~$g_a$ and~$g_b^{\sigma}$ can have at most three 
roots in common. We show now that with this choice of~$V_a$ and~$V_b$ there are 
finitely many $\sigma \in \pgl(2,\C)$ such that $\dim ( V_a + V_b^{\sigma} ) 
\leq c+1$. Notice that, taking into account 
Equation~\eqref{equation:finiteness_condition_caseA}, the latter condition is 
equivalent to
\[
 \dim ( V_a \cap V_b^{\sigma} ) \, \geq \, a + b - c + 1 \, = \, 1.
\]
Moreover, the set $V_a \cap V_b^{\sigma}$ is constituted of the 
multiples of degree $d$ of the least common multiple of~$g_a$ 
and~$g_b^{\sigma}$. Since by 
Equation~\eqref{equation:finiteness_condition_caseA} we have $(d-a) + (d-b) = 
d+3$, it follows
\[
 \deg \Bigl( \operatorname{lcm}\bigl(g_a, g_b^{\sigma}\bigr)\Bigr) = d + 3 - 
 \deg \Bigl( \operatorname{gcd}\bigl(g_a, g_b^{\sigma}\bigr)\Bigr).
\]
Taking into account that by the cross-ratio hypothesis we have 
\[
 \deg \Bigl( \operatorname{gcd}\bigl(g_a, g_b^{\sigma}\bigr)\Bigr) \, 
 \leq \, 3,
\]
it follows that the only elements $\sigma \in \pgl(2,\C)$ for which the 
condition $\dim ( V_a \cap V_b^{\sigma} ) \geq 1$ is satisfied are the ones such 
that $g_a$ and $g_b^{\sigma}$ have exactly three roots in common. By the 
cross-ratio hypothesis, and by the fact that every automorphism of~$\p^1$ is 
completely determined by the images of three projectively independent points, it 
follows that there are only finitely many such~$\sigma$. Moreover, we can also 
count the number of these automorphisms~$\sigma$: each of them is specified by a 
pair constituted of a triple of roots of~$g_a$ and a triple of roots of~$g_b$, 
so in total they are (remember that all roots of~$g_a$ and~$g_b$ are distinct 
because of the squarefreeness hypothesis)
\[
 6 \, \binom{d-a}{3} \binom{d-b}{3} \, = \, 6 \, \binom{a+3}{3} \binom{b+3}{3}.
\]
In the next section we prove that this is also the number when we 
take~$V_a$ and~$V_b$ to be general. This proves our initial claim when
Equation~\eqref{equation:finiteness_condition_caseA} holds.

In order to obtain an example of a point in~$\Gr{a}{d} \times \Gr{b}{d}$ whose 
preimage under~$\phi$ is finite when 
Equation~\eqref{equation:finiteness_condition_caseB} holds, we employ the 
perfect pairing~\eqref{equation:pairing} introduced in 
Section~\ref{linearization}. The orthogonals~$V_a^{\perp}$ and~$V_b^{\perp}$ of 
the spaces~$V_a$ and~$V_b$ with respect to this pairing have 
dimension~$d-a$ and~$d-b$, respectively. By unraveling the definitions of the 
vector subspaces involved, and using the bijection~\eqref{equation:bijection}, 
it 
follows that there is a bijection between the preimages of the points $([V_a], 
[V_b])$ and $([V_a^{\perp}], [V_b^{\perp}])$ --- which belong to different 
Grassmannians and so appear in different instances of our problem. Moreover, if 
with $([V_a], [V_b])$ we are in the situation prescribed by 
Equation~\eqref{equation:finiteness_condition_caseA}, then with $([V_a^{\perp}], 
[V_b^{\perp}])$ we are in the situation prescribed by 
Equation~\eqref{equation:finiteness_condition_caseB}. This shows that we do not 
need to provide another example to ensure that also when
Equation~\eqref{equation:finiteness_condition_caseB} holds the map~$\phi$ is 
dominant, and so our initial claim is proven.

\section{A formula for the number of solutions}
\label{formula}

In this section we compute the number of solutions of our problem when the two 
vector subspaces $V_a$ and $V_b$ are general. From the previous section we know 
that this number is finite when Equation~\eqref{equation:finiteness_condition} 
 holds. 

We associate to~$V_b$ a rational map $\Phi_{V_b} \colon 
\p^3 \dashrightarrow \Gr{b}{d}$ and to~$V_a$ a Schubert variety~$Z_{V_a}$ 
in~$\Gr{b}{d}$ so that the number we are looking for 
is the degree of~$\Phi_{V_b}^{-1}\bigl(Z_{V_a}\bigr)$. The difficulty we 
encounter here in computing such degree is that $\Phi_{V_b}$ is not defined on 
the whole~$\p^3$, hence we lift it to a morphism $\tilde{\Phi}_{V_b} \colon 
\tilde{\p^3} \longrightarrow \Gr{b}{d}$, where $\tilde{\p^3}$ is a 
blow up of~$\p^3$. By doing this and using Porteous-Giambelli formula, 
everything is reduced to the computation of the Chern classes of a vector 
bundle on~$\tilde{\p^3}$. In particular, we will be interested in the degree of 
some polynomial combination of these Chern classes, and we will calculate these 
degrees via Bott residue formula.

We start with the construction of~$\Phi_{V_b}$. Define $U$ to be the complement 
of the quadric $\{ \alpha \delta - \beta 
\gamma = 0 \}$ in~$\p^3$, where we take projective coordinates $(\alpha: 
\beta: \gamma: \delta)$. We identify a matrix $A = \left( \begin{smallmatrix} 
\alpha & \beta \\ \gamma & \delta \end{smallmatrix} \right)$ with the point 
$(\alpha: \beta: \gamma: \delta)$, so that $U$ is in bijection with $\pgl(2, 
\C)$. The morphism $\Phi_{V_b} \colon U \longrightarrow \Gr{b}{d}$ is defined 
as follows: to every point $A \in U$, we set $\Phi_{V_b}(A)$ 
to be the point in $\Gr{b}{d}$ associated to the vector subspace~$V_b^\sigma$, 
where $\sigma \in \pgl(2,\C)$ is the automorphism corresponding to~$A$. It is 
known (see~\cite[Proposition~2.8]{Arrondo1996}) 
that every morphism to a Grassmannian is completely determined by a 
locally free sheaf, together with a choice of a vector subspace of its global 
sections.
In our case, the locally free sheaf~$\mscr{M}$ on~$U$ determining~$\Phi_{V_b}$ 
is generated by the columns of the matrix~$\mcal{M}$ with entries in $R = 
\C[\alpha, \beta, \gamma, \delta]$ obtained in the following way. Let $(G_0, 
\dotsc, G_b)$ be a basis for~$V_b$, then the $i$-th row of~$\mcal{M}$ is given 
by extracting the coefficients of $G_i$ (with respect to the monomial basis 
$s^d, \dotsc, t^d$ of~$\C[s,t]_d$) after having applied to it the change of 
variables 
\[
 \begin{pmatrix} s \\ t \end{pmatrix} \mapsto 
 \begin{pmatrix} \alpha & \beta \\ \gamma & \delta \end{pmatrix} 
 \begin{pmatrix} s \\ t \end{pmatrix}.
\]
If we denote by~$M$ the graded submodule of~$R(d)^{b+1}$ generated by the 
columns of~$\mcal{M}$, then $\mscr{M} = \widetilde{M}$. Here $R(b)$ denotes the 
graded $R$-module obtained by shifting by~$b$ the standard 
$\Z$-grading on~$R$. The locally free sheaf~$\mscr{M}$ is then the restriction 
to~$U$ of a coherent sheaf on~$\p^3$, which we still denote by~$\mscr{M}$. 

Notice that the group $\pgl(2,\C)$ acts both from the left and from the right 
on~$\p^3$ by left and right matrix multiplication, when we identify points 
in~$\p^3$ with equivalence classes of $2 \times 2$ matrices. Both these actions 
induce naturally actions on~$R(d)^{b+1}$. The right 
action will play a crucial in the following, in particular in the proof of 
Proposition~\ref{proposition:bad_lines} and in 
Propositions~\ref{proposition:description_bundle} 
and~\ref{proposition:degrees}.

\begin{lemma}
\label{lemma:right_action}
The right action of~$\pgl(2, \C)$ on~$R(d)^{b+1}$ preserves the submodule~$M$.
\end{lemma}
\begin{proof}
The right action by an element of~$\pgl(2,\C)$ sends each of the generators 
of~$M$ to a linear combination (with complex coefficients) of the same 
generators. This can be immediately seen if we write the matrix $\mcal{M}$ as 
the product $\mcal{K} \cdot 
\operatorname{Sym}_d \left( \begin{smallmatrix} \alpha & \beta \\ \gamma & 
\delta \end{smallmatrix} \right)$, where $\mcal{K}$ is a $(b+1) \times 
(d+1)$ matrix with complex coefficients, while $\operatorname{Sym}_d(\cdot)$ 
denotes the $d$-th symmetric algebra functor. In this way, if $\sigma \in 
\pgl(2, \C)$ is any element, then the generators of the module $M \cdot 
\sigma$ are the columns of the matrix $\mcal{K} \cdot 
\operatorname{Sym}_d \left( \begin{smallmatrix} \alpha & \beta \\ \gamma & 
\delta \end{smallmatrix} \right) \cdot\operatorname{Sym}_d (\sigma)$, which are 
complex linear combinations of the columns of~$\mcal{M}$.
\end{proof}

\begin{remark}
\label{remark:preserved_lines}
 Consider the two rulings of the quadric $\{\alpha \delta - \beta \gamma = 0\}$ 
in~$\p^3$. A direct computation shows that the right action of~$\pgl(2,\C)$ 
preserves one of the two rulings as a whole, permuting the lines in this 
ruling, while it preserves each individual line of the other.
\end{remark}

A nice property of the coherent sheaf~$\mscr{M}$ is that it is locally free on 
an open subset larger than~$U$ and whose complement in~$\p^3$ is constituted by 
disjoint lines. We prove this in 
Proposition~\ref{proposition:bad_lines}. To this end, we introduce some 
technical tools.

Starting from~$V_b$, we define a zero-dimensional subscheme in~$\p^1$ that will 
be used to describe the locus in~$\p^3$ where $\mscr{M}$ is not locally free.

\begin{definition}
\label{definition:bad_points}
 Let $V_b \subseteq \C[s,t]_d$ be a vector subspace of dimension~$b+1$. Define
 \[
  \widehat{B} \, := \, 
  \bigl\{ 
   (s:t) \in \p^1 \, : \, \text{there exists } 
   G \in V_b \text{ such that } \operatorname{ord}_{(s:t)} G \geq b+1 
  \bigr\}.
\]
If $G_0, \dotsc, G_b$ is a basis for~$V_b$, then we have
\begin{equation}
\label{equation:B_scheme}
  \widehat{B} \, = \, \left\{ (s:t) \in \p^1 \, : \, \left| 
  \begin{array}{*3{>{\displaystyle}c}} \displaystyle
    \frac{\partial^{b}G_0}{\partial s^b}(s,t) & \cdots & 
    \frac{\partial^{b}G_0}{\partial t^b}(s,t) \\
    \vdots & \ddots & \vdots \\
    \frac{\partial^{b}G_b}{\partial s^b}(s,t) & \cdots & 
    \frac{\partial^{b}G_b}{\partial t^{b}}(s,t)
  \end{array}
  \right| = 0 \right\}.
\end{equation}
Notice that the last description of~$\widehat{B}$ allows to endow it with the 
structure of a scheme. Using Euler's identity for homogeneous 
polynomials, and column reduction, one sees that the determinant 
in Equation~\eqref{equation:B_scheme}, multiplied by a suitable power 
of~$s$, is a scalar multiple of
\begin{equation}
\label{equation:wronskian}
 \left| 
  \begin{array}{*4{>{\displaystyle}c}} 
    G_0(s,t) & \frac{\partial G_0}{\partial t}(s,t)  & \cdots & 
\frac{\partial^{b}G_0}{\partial t^b}(s,t) \\
    & \vdots &  & \vdots \\
    G_b(s,t) & \frac{\partial G_0}{\partial t}(s,t) & \cdots & 
\frac{\partial^{b}G_b}{\partial t^{b}}(s,t)
  \end{array}
  \right|.
\end{equation}
If the determinant in Equation~\eqref{equation:wronskian} is identically zero, 
then in particular it is so when $s=1$. In this case, however, in 
Equation~\eqref{equation:wronskian} we get the Wronskian of the polynomials 
$G_0, \dotsc, G_b$, and the Wronskian of linearly independent polynomials 
cannot 
be identically zero (see~\cite{Bostan2010}). Hence it cannot happen 
that $\widehat{B} = \p^1$. Since the determinant 
in Equation~\eqref{equation:B_scheme} has degree~$(b+1)(d-b)$, then 
$\widehat{B}$ is a zero-dimensional scheme of length~$(b+1)(d-b)$.
\end{definition}

\begin{lemma}
\label{lemma:general_bad_points}
 Let $V_b \subseteq \C[s,t]_d$ be a general vector subspace of dimension~$b+1$. 
Then the set~$\widehat{B}$ as in Definition~\ref{definition:bad_points} is 
constituted of $(b+1)(d-b)$ distinct points.
\end{lemma}
\begin{proof}
 Denote by~$\operatorname{Hilb}_{(b+1)(d-b)}(\p^1)$ the Hilbert scheme of 
zero-dimensional schemes of length~$(b+1)(d-b)$ in~$\p^1$. Since it is the 
projectivization of the space of bivariate polynomials of 
degree~$(b+1)(d-b)$, the scheme~$\operatorname{Hilb}_{(b+1)(d-b)}(\p^1)$ 
is irreducible of dimension~$(b+1)(d-b)$. Consider now the morphism:
\[
 \begin{array}{rccc}
  \Xi \colon & \Gr{b}{d} & \longrightarrow & 
  \operatorname{Hilb}_{(b+1)(d-b)}(\p^1) \\
  & [V_b] & \mapsto & [\widehat{B}]
 \end{array}
\]
where $\widehat{B}$ is as in Definition~\ref{definition:bad_points}. By 
construction, the map $\Xi$ is a morphism between varieties of the same 
dimension. If we show that $\Xi$ is dominant, then the statement is proven, 
since the locus of schemes constituted of $(b+1)(d-b)$ distinct points is open 
in~$\operatorname{Hilb}_{(b+1)(d-b)}(\p^1)$. 

We prove that $\Xi$ is dominant as in Section~\ref{finiteness}, namely by 
showing that $\Xi^{-1}([Z])$ is zero-dimensional for a particular point $[Z] \in 
\operatorname{Hilb}_{(b+1)(d-b)}(\p^1)$. We pick $Z$ to be the subscheme 
in~$\p^1$ defined by the ideal~$\bigl( t^{(b+1)(d-b)} \bigr)$, where we take 
$(s:t)$ as homogeneous coordinates in~$\p^1$. A direct computation shows that if 
we take $\overline{V}_b := \left\langle s^b t^{d-b}, \dotsc, t^d \right\rangle$, 
then $\Xi([\overline{V}_b]) = [Z]$. Hence $[Z]$ is in the image of~$\Xi$. Now 
suppose that $\Xi([V_b]) = [Z]$ for some vector subspace $V_b = \left\langle 
G_0, \dotsc, G_b \right\rangle$. We are going to show that $\ord_{(1:0)} G_i 
\geq d-b$ for all $i \in \{0, \dotsc, b\}$. If this is true, then $V_b$ is 
contained in the $d$-th homogeneous component~$(t^{d-b})_d$ of the 
ideal~$(t^{d-b})$. Since both~$V_b$ and~$(t^{d-b})_d$ have dimension~$b+1$, they 
are equal. This proves that, at least set-theoretically, the fiber 
$\Xi^{-1}([Z])$ is constituted of a single point, so in particular it is 
zero-dimensional. This concludes the proof. 

To show that the order of the polynomials~$\{ G_i \}_{i=0}^{b}$ is at 
least~$d-b$ at~$(1:0)$, let $\alpha_i := \ord_{(1:0)} G_i$ for all~$i$.
Notice that we can suppose that all the orders of the~$G_i$ are different, and 
we can order them so that
\[
 \alpha_0 \lneq \alpha_1 \lneq \dotsb \lneq \alpha_b.
\]
Because of Lemma~\ref{lemma:no_cancellation} below, we know that $\sum 
(\alpha_i - i) = (b+1)(d-b)$. This forces $\alpha_i \geq d-b$ for all~$i$: in 
fact, we have $\alpha_b \leq d$, so $\alpha_b - b \leq d-b$, hence $\alpha_i - 
i \leq d-b$ for all~$i$, and so actually we must have $\alpha_i - i = d-b$ for 
all~$i$.
\end{proof}

We thank Christoph Koutschan for providing us the proof of the following lemma.

\begin{lemma}
\label{lemma:no_cancellation}
 Let $G_0, \dotsc, G_b$ be homogeneous polynomials in~$\C[s,t]_d$, and let 
$\alpha_i := \ord_{(1:0)} G_i$. Suppose that $\alpha_0 \lneq \alpha_1 \lneq 
\dotsb \lneq \alpha_b$ holds. Then the order at~$(1:0)$ of the determinant of 
the matrix from Equation~\eqref{equation:wronskian} equals $\sum (\alpha_i - 
i)$.
\end{lemma}
\begin{proof}
By a direct inspection of the orders of the entries of the matrix in 
Equation~\eqref{equation:wronskian}, we see that the order of the determinant 
is at least $\sum (\alpha_i - i)$, and it is exactly equal to 
this number if no cancellation occurs when we compute the determinant using the 
standard Leibniz formula. One sees that it is harmless to set $s 
= 1$, and that we can suppose that the coefficient of the monomial 
$t^{\alpha_i}$ in $G_i(1,t)$ is~$1$. Everything reduces to show that 
the matrix obtained by taking the trailing coefficients of the entries (namely, 
the coefficients of the monomials where the lowest power of~$t$ appears) is 
non-singular. This matrix is the following:
\begin{equation}
\label{equation:vandermonde}
 \left| 
  \begin{array}{*5{>{\displaystyle}c}} 
    1 & \alpha_0  & \alpha_0(\alpha_0 - 1) & \cdots & 
    \alpha_0(\alpha_0 - 1) \cdots (\alpha_0 - b + 1) \\
    \vdots & \vdots & \vdots & & \vdots \\
    1 & \alpha_b & \alpha_b(\alpha_b - 1) & \cdots & 
    \alpha_b(\alpha_b - 1) \cdots (\alpha_b - b + 1)
  \end{array}
  \right|.
\end{equation}
By applying column reduction, the determinant in 
Equation~\eqref{equation:vandermonde} turns out to be equal to the Vandermonde 
determinant $\prod_{i < j} (\alpha_i - \alpha_j)$, which is not zero since all 
the $\{ \alpha_i \}_{i=0}^{b}$ are different.
\end{proof} 

\begin{lemma}
\label{lemma:general_orders}
 Let $V_b \subseteq \C[s,t]_d$ be a general vector subspace of 
dimension~\mbox{$b+1$}. 
Consider a point $(\bar{s}:\bar{t}) \in \p^1$ and let $G_0, \dotsc, G_b$ be a 
basis for~$V_b$ such that $\bigl( \ord_{(\bar{s}:\bar{t})}G_i \bigr)_i$ is a 
strictly increasing sequence. Let  $\widehat{B}$ be as in 
Definition~\ref{definition:bad_points}. Then
\begin{itemize}
 \item 
  if $(\bar{s}:\bar{t}) \not\in \widehat{B}$, we have $\ord_{(\bar{s}:\bar{t})} 
G_i = i$ for all~$i$;
 \item
  if $(\bar{s}:\bar{t}) \in \widehat{B}$, we have $\ord_{(\bar{s}:\bar{t})} G_i 
= i$ for $i \in \{0, \dotsc, b-1\}$ and $\ord_{(\bar{s}:\bar{t})}G_b = b+1$.
\end{itemize}
\end{lemma}
\begin{proof}
 Suppose that $(\bar{s}:\bar{t}) \not\in \widehat{B}$. Then $(\bar{s}:\bar{t})$ 
is not a zero of the determinant in Equation~\eqref{equation:wronskian}, so the 
order of this determinant at~$(\bar{s}:\bar{t})$ is zero. By 
Lemma~\ref{lemma:no_cancellation} it follows that $\ord_{(\bar{s}:\bar{t})} G_i 
= i$ for all~$i$. 

Suppose now that $(\bar{s}:\bar{t}) \in \widehat{B}$. Then by 
Lemma~\ref{lemma:general_bad_points} the determinant in 
Equation~\eqref{equation:wronskian} has order~$1$ at~$(\bar{s}:\bar{t})$. Again 
by Lemma~\ref{lemma:no_cancellation} it follows that exactly one of the numbers 
$\{ \ord_{(\bar{s}:\bar{t})} G_i - i\}_{i=0}^{b}$ equals~$1$, while all the 
others are zero. By hypothesis we have 
\[
 \ord_{(\bar{s}:\bar{t})} G_0 - 0 \leq 
 \ord_{(\bar{s}:\bar{t})} G_1 - 1 \leq 
 \dotsb \leq 
 \ord_{(\bar{s}:\bar{t})} G_b - b,
\]
and so the only possibility is the one presented in the statement.
\end{proof}

\begin{proposition}
\label{proposition:bad_lines}
For a general choice of a subspace~$V_b \subseteq \C[s,t]_d$ of 
dimension~\mbox{$b+1$}, the sheaf~$\mscr{M}$ is locally free on an open set $U' 
\supseteq U$ that is the complement of $(b + 1)(d - b)$ disjoint lines 
in~$\p^3$.
\end{proposition}
\begin{proof}
We define the set
\[
  B \, := \, 
  \left\{ 
   \left(
   \begin{array}{cc}
     xu & xv \\ yu & yv 
   \end{array}
   \right) \in \p^3 \, : \, 
   (x:y) \in \widehat{B}, \ (u:v) \in \p^1 
  \right\},
\]
where $\widehat{B}$ is as in Definition~\ref{definition:bad_points}. 
Lemma~\ref{lemma:general_bad_points} implies that $B$ is a set of \mbox{$(b + 
1)$}\mbox{$(d - b)$} disjoint lines in~$\p^3$. In the future, we will use the 
fact that each of these lines is preserved by the right action of~$\pgl(2,\C)$ 
(see Remark~\ref{remark:preserved_lines}). We define the open set $U'$ to be 
the complement of~$B$ in~$\p^3$, so by construction we have $U \subseteq U'$.

We prove that $\mscr{M}$ is locally free at every point in $U'$. As far as the 
points in~$U$ are concerned, there is nothing to prove. Let $A \in U' 
\setminus U$, then we use the left and the right action of~$\pgl(2, \C)$ 
on~$\p^3$ and suppose that $A = \left( \begin{smallmatrix} 1 & 0 \\ 0 & 0 
\end{smallmatrix} \right)$. Hence, it follows that with this choice of 
coordinates the point $(1:0)$ does not belong to~$\widehat{B}$, because 
otherwise we would have $A \in B$. Notice that the left action does not 
preserve the module~$M$ but this is not a problem, since we only want to 
establish local freeness, and so we can also work with modules that are just 
isomorphic to~$M$. By Lemma~\ref{lemma:general_orders}, we can choose 
a basis $G_0, \dotsc, G_b$ of~$V$ with $\operatorname{ord}_{(1:0)} G_i = i$ 
for all~$i$. In this way, the matrix~$\mcal{M}$ whose columns 
generate~$\mscr{M}$ has the form 
\[ 
  \mcal{M}_{ij} \, = \, 
  \frac{\partial^{d} G_{i}(\alpha s + \beta t, \gamma s + \delta t)}{\partial 
s^{d-j} \partial t^{j}} \,.
\]
Since the question is local, we can restrict ourselves to the open chart 
of~$\p^3$ where $\alpha = 1$. This corresponds to consider $R' = 
\C[\beta, \gamma, \delta]$, the coordinate ring of the open chart of~$\p^3$ we 
are working on, and the restriction of~$\mscr{M}$ to such chart, whose 
corresponding module~$M'$ is generated by the columns of~$\mcal{M}$ where we 
make the substitution $\alpha=1$. Our goal is to prove that the 
first~$b+1$ columns of~$\mcal{M}$ generate freely~$M'$ over the 
ring~$R'_{\mscr{A}}$, where $\mscr{A}$ is the maximal ideal in~$R'$ of the 
point~$A$. To make the computations easier, we employ the substitution $s \to s 
- \beta t$: such a substitution operates on the matrix~$M'$ as the 
multiplication on the left by an invertible $(d+1) \times (d+1)$ matrix with 
entries in $R'$; hence the modules spanned by the columns of these two 
matrices are isomorphic. Thus, for our purposes we can suppose 
that the matrix~$\mcal{M}$ has the form $\mcal{M}_{ij} = 
\frac{\partial^{d} G_{i}(s, \gamma s  - \gamma \beta t + 
\delta t)}{\partial s^{d-j} \partial t^{j}}$. Eventually, we can perform the
change of variables $D = \delta - \beta \gamma$, obtaining 
\[
  \mcal{M}_{ij} \, = \,
  \frac{\partial^{d} G_{i}(s, \gamma s  + Dt)}{\partial s^{d-j} \partial t^{j}}.
\]
To write explicitly the entries $\mcal{M}_{ij}$ we employ the following Taylor
expansion --- here we write the expansion $T(1) = T(0) + T'(0) + \dotsb$, 
where $T(z) := G_{i}\bigl(s, \gamma s + Dt \tth z \bigr)$:
\[
  \begin{split}
    G_{i}\bigl(s , \gamma s + D t\bigr) & = 
\sum_{j=0}^{d} \frac{(Dt)^{j}}{j!} \cdot \frac{\partial^j G_i}{\partial 
t^j} (s, \gamma s) \\
    & = \sum_{j=0}^{d} \frac{(Dt)^{j}}{j!} \cdot s^{d-j} \cdot
\frac{\partial^j G_i}{\partial t^j} (1, \gamma).
  \end{split}
\]
Hence $\mcal{M}_{ij} = \frac{D^j}{j!} \frac{\partial^j G_i}{\partial t^j}
(1, \gamma)$. In particular, we see that $D^{j}$ divides
$\mcal{M}_{ij}$ for all $i,j$. Moreover, $\mcal{M}_{00} = G_0(1, \gamma)$ and
thus $\mcal{M}_{00}$ does not vanish on~$A$, since by hypothesis $G_0(1,0) \neq
0$. Therefore $\mcal{M}_{00}$ is invertible in~$R'_{\mscr{A}}$, so we
can perform row reductions on~$\mcal{M}$ over~$R'_{\mscr{A}}$, obtaining a new
matrix $\overline{\mcal{M}}$ whose columns generate, over~$R'_{\mscr{A}}$, a 
module isomorphic to the one generated by the the columns of~$\mcal{M}$:
\[ 
  \begin{array}{ll}
    \overline{\mcal{M}}_{0j} := \mcal{M}_{0j} & \text{for all } j, \\
    \overline{\mcal{M}}_{ij} := \mcal{M}_{ij} - 
\frac{\mcal{M}_{i0}}{\mcal{M}_{00}} \mcal{M}_{0j} & \text{for all } j, \ 
\text{for all } i \geq 1.
  \end{array}
\]
\[
  \overline{\mcal{M}} \, = \, \left(
  \begin{array}{ccccc}
    \overline{\mcal{M}}_{00} & D \cdot \nicefrac{\overline{\mcal{M}}_{01}}{D} &
D^2 \cdot \nicefrac{\overline{\mcal{M}}_{02}}{D^2} & \cdots & D^d \cdot
\nicefrac{\overline{\mcal{M}}_{0d}}{D^d} \\
    0 &  D \cdot \nicefrac{\overline{\mcal{M}}_{11}}{D}  & D^2 \cdot
\nicefrac{\overline{\mcal{M}}_{12}}{D^2} & \cdots & D^d \cdot
\nicefrac{\overline{\mcal{M}}_{1d}}{D^d} \\
    \vdots \\
    0 &  D \cdot \nicefrac{\overline{\mcal{M}}_{b1}}{D} & D^2 \cdot
\nicefrac{\overline{\mcal{M}}_{b2}}{D^2} & \cdots & D^d \cdot
\nicefrac{\overline{\mcal{M}}_{bd}}{D^d}
  \end{array} \right).
\]
Notice that $\overline{\mcal{M}}_{11} = D \cdot \left( \frac{\partial
G_1}{\partial t}(1, \gamma) - \frac{G_1(1,\gamma)}{G_0(1, \gamma)}
\frac{\partial G_0}{\partial t}(1, \gamma) \right)$. In particular,
$\overline{\mcal{M}}_{11} / D$ does not vanish at~$A$, since by hypothesis
$G_1(1,0) = 0$ and $\frac{\partial G_1}{\partial t}(1,0) \neq 0$. Now observe
that the $R'_{\mscr{A}}$-module generated by the columns
of~$\overline{\mcal{M}}$ is isomorphic to the one generated by the columns of
\[
  \left( \begin{array}{ccccc}
    \overline{\mcal{M}}_{00} & D \cdot \nicefrac{\overline{\mcal{M}}_{01}}{D} &
D^2 \cdot \nicefrac{\overline{\mcal{M}}_{02}}{D^2} & \cdots & D^d \cdot
\nicefrac{\overline{\mcal{M}}_{0d}}{D^d} \\
    0 &  \nicefrac{\overline{\mcal{M}}_{11}}{D}  & D \cdot
\nicefrac{\overline{\mcal{M}}_{12}}{D^2} & \cdots & D^{d-1} \cdot
\nicefrac{\overline{\mcal{M}}_{1d}}{D^d} \\
    \vdots \\
    0 &  \nicefrac{\overline{\mcal{M}}_{b1}}{D} & D \cdot
\nicefrac{\overline{\mcal{M}}_{b2}}{D^2} & \cdots & D^{d-1} \cdot
\nicefrac{\overline{\mcal{M}}_{bd}}{D^d}
  \end{array} \right),
\]
where we divided by~$D$ all the rows from the second to the $(b+1)$-th.
At this point we can repeat the Gaussian elimination using the second row, and
then ``divide'' again by~$D$ all the rows from the third to the $(b+1)$-th.
The hypothesis on~$A$ and on the basis $G_0, \dotsc, G_b$ ensures that this 
process can be carried over for all rows. In this way we eventually achieve an 
echelonized form the matrix, which shows that the first $b+1$ columns generate 
freely a module isomorphic to~$M'$. Hence $M$ itself is locally free at~$A$, 
this proving the claim.
\end{proof}

\begin{definition}
\label{definition:schubert_cycle}
Let $V_a$ be a vector subspace of~$\C[s,t]_d$ of dimension~$a+1$ and let $c 
\in \N$. We consider the following subvariety of the Grassmannian~$\Gr{b}{d}$, 
which by construction is a Schubert subvariety:
\[
  Z_{V_a} \, := \, 
  \bigl\{ 
   \Lambda \in \Gr{b}{d} \, : \, 
   \mathrm{dim}(V_a + \Lambda) \leq c+1 
  \bigr\}.
\]
\end{definition}

We are interested in the cardinality of the 
set~$\left(\Phi_{V_b}\right)^{-1}\bigl(Z_{V_a}\bigr)$, which we proved 
to be finite in Section~\ref{finiteness} when $V_a$ and $V_b$ are general. In 
order to compute such number, we could use the machinery of intersection theory, 
in particular Porteous-Giambelli theorem. Unfortunately, the domain~$U$ of the 
regular map~$\Phi_{V_b}$ is not a projective variety, so the result cannot be 
applied directly.

On the other hand, the morphism~$\Phi_{V_b}$ gives a rational map on~$\p^3$, 
which we still denote by~$\Phi_{V_b}$. By what we proved in 
Proposition~\ref{proposition:bad_lines}, the locus~$U'$ where $\Phi_{V_b}$ is 
regular is bigger than~$U$ and it is the complement of a number of disjoint 
lines. However, notice that only the points in~$U$ correspond to automorphisms 
of~$\p^1$, hence they are the only ones to be considered in solving our initial 
problem.  From the theorem of resolution of 
indeterminacies of a rational map (see~\cite[Section~4.2 and 
Lemma~4.8]{Cutkosky2004}), we know that there exists a scheme~$\tilde{\p^3}$, a 
morphism $\tau \colon \tilde{\p^3} \longrightarrow \p^3$ and a morphism 
$\tilde{\Phi}_{V_b} \colon \tilde{\p^3} \longrightarrow \Gr{b}{d}$ making the 
following diagram commutative:
\[
  \xymatrix@C=4pc{\tilde{\p^3} \ar[rd]^-{\tilde{\Phi}_{V_b}} \ar[d]_{\tau} \\ 
  \p^3 \ar@{-->}[r]_-{\Phi_{V_b}} & \Gr{b}{d}}
\]
We are going to show that we can use the map~$\tilde{\Phi}_{V_b}$ to calculate 
the desired number. To do so, we have to exclude first that some point in 
$\tilde{\Phi}^{-1}_{V_b}(Z_{V_a})$ lies in $\tilde{\p^3} \setminus 
\tau^{-1}(U)$, namely does not correspond to an automorphism of~$\p^1$. In 
Lemma~\ref{lemma:general_preimage} we show that if $V_a$ and $V_b$ are general, 
then $\tilde{\Phi}_{V_b}^{-1}\bigl(Z_{V_a}\bigr)$ is always completely contained 
in~$\tau^{-1}(U)$.

\begin{lemma}
\label{lemma:general_preimage}
 Let $V_a, V_b \subseteq \C[s,t]_d$ be general vector subspaces of 
dimension~$a+1$ and~$b+1$, respectively. Suppose that 
Equation~\eqref{equation:finiteness_condition} holds. Then the preimage 
of~$Z_{V_a}$ under~$\tilde{\Phi}_{V_b}$ is contained in~$\tau^{-1}(U)$.
\end{lemma}
\begin{proof}
 We have to show that no element in $\tilde{\p^3} \setminus \tau^{-1}(U)$ 
belongs to $\tilde{\Phi}_{V_b}^{-1}\bigl(Z_{V_a}\bigr)$. Let $W$ be the Zariski 
closure of $\tilde{Phi}_{V_b}\bigl(\tilde{\p^3} \setminus \tau^{-1}(U)\bigr)$ 
in~$\Gr{b}{d}$, with the reduced structure. Notice that the irreducible 
components of~$W$ are integral subschemes of~$\Gr{b}{d}$ of dimension at 
most~$2$, since $\tilde{\p^3} \setminus \tau^{-1}(U)$ is contained in the 
preimage via~$\tau$ of the quadric $\{ \alpha \delta - \beta \gamma = 0\}$ 
in~$\p^3$. We show that, since $V_a$ is general, it is always possible to 
avoid~$W$ with~$Z_{V_a}$. Notice that the algebraic group~$\pgl(d+1, \C)$ acts 
transitively on~$\Gr{b}{d}$ via its standard action on~$\p^d$; if $g \in 
\pgl(d+1, \C)$, we denote by~$g \cdot Z_{V_a}$ the translate of~$Z_{V_a}$ under 
the action of~$g$. Moreover, a computation similar to the one providing 
Equation~\eqref{equation:codimension} shows that $Z_{V_a}$ has codimension 
$(d-c)(a+b-c+1)$, which equals~$3$ because we suppose that 
Equation~\eqref{equation:finiteness_condition} holds. Then, by Kleiman's 
transversality theorem \cite[Corollary~4]{Kleiman1974}, the dimension of the 
intersection $(g \cdot Z_{V_a}) \cap W$ is $-1$ for every~$g$ belonging to an 
open subset of~$\pgl(d+1, \C)$. This means that if $V_a$ is general, then we can 
suppose that $Z_{V_a}$ and $W$ do not intersect, and this concludes the proof. 
\end{proof}

Since, as it will be made clear in 
Proposition~\ref{proposition:description_bundle}, the morphism~$\tau$ is an 
isomorphism outside the indeterminacy locus of~$\Phi_{V_b}$, then the 
cardinality of~$\Phi_{V_b}^{-1}(Z_{V_a})$ equals the cardinality 
of~$\tilde{\Phi}_{V_b}^{-1}(Z_{V_a})$. Moreover, a by-product of 
Section~\ref{finiteness} is that, by the fact that $V_a$ and $V_b$ are general, 
the fiber~$\Phi_{V_b}^{-1}(Z_{V_a})$ is constituted of smooth points. Hence, to 
compute the number we are interested in, it is enough to compute the degree of 
the $0$-cycle~$\tilde{\Phi}_{V_b}^{*}([Z_{V_a}])$, where $[Z_{V_a}]$ is the 
Schubert cycle given by the Schubert variety~$Z_{V_a}$.

In order to compute this degree, let us start by noticing that we can 
express~$[Z_{V_a}]$ in terms of the universal bundle~$\mscr{U}^{*}$ of the 
Grassmannian~$\Gr{b}{d}$. In fact, if $h_1, \dotsc, h_{d-a}$ are linear forms 
defining~$V_a$, we can interpret them as global sections 
of~$\mscr{U}^{*}$, and then $Z_{V_a}$ becomes the locus where the sections 
$h_1, \dotsc, h_{d-a}$ have rank at most~$c-a-1$ (see 
\cite[Example~4.7]{Arrondo2010}). Now Schubert calculus tells us (see 
\cite[Example~4.9]{Arrondo2010}) that the cycle associated to this locus is 
given by the determinant
\begin{equation}
\label{equation:determinant_schubert}
 \resizebox{.92\textwidth}{!}{$
 \left|
 \!\!
 \begin{array}{cccc}
  \chern{(b+1)-(c-a)}{\mscr{U}^{*}} & 
   \chern{(b+2) - (c-a)}{\mscr{U}^{*}} & 
   \cdots & \chern{(b+d-a) - 2(c-a)}{\mscr{U}^{*}} \\
  \chern{b-(c-a)}{\mscr{U}^{*}} & 
   \chern{(b+1) - (c-a-1)}{\mscr{U}^{*}} & 
   \cdots & \chern{(b+d-a) - 2(c-a)+1}{\mscr{U}^{*}} \\
  \vdots \\
  \chern{b-(d-a)+2}{\mscr{U}^{*}} & 
   \chern{b-(d-a)+3}{\mscr{U}^{*}} & 
   \cdots & \chern{(b+1) - (c-a)}{\mscr{U}^{*}} \\
 \end{array}
 \!\!
 \right|
 $}
\end{equation}
where $\chern{i}{\mscr{U}^{*}}$ denotes the $i$-th Chern class of the vector 
bundle~$\mscr{U}^{*}$. 

Since $\tilde{\Phi}_{V_b}$ is a morphism to a Grassmannian, by what we 
reported at the beginning of the section there exists a vector 
bundle~$\mscr{Q}$ on~$\tilde{\p^3}$ and a vector subspace of sections 
of~$\mscr{Q}$ such that the morphism they induce is~$\tilde{\Phi}_{V_b}$. It 
follows from the functoriality of the pullback that the 
cycle~$\tilde{\Phi}_{V_b}^{*}([Z_{V_a}])$ is given by the determinant in 
Equation~\eqref{equation:determinant_schubert}, after 
substituting~$\mscr{U}^{*}$ with~$\mscr{Q}$. This is the so-called 
Porteous-Giambelli formula, see \cite[Theorem~3.10]{Arrondo2010} and 
\cite[Corollary~6]{Kempf1974}.

Using a ``Gaussian reduction'' technique as in 
Proposition~\ref{proposition:bad_lines}, we 
can express in Proposition~\ref{proposition:description_bundle} the 
variety~$\tilde{\p^3}$ and the vector bundle~$\mscr{Q}$ in a very concrete 
way. We proceed following closely the technique for the resolution of 
indeterminacies in the case of maps to a projective space described in 
\cite[Chapter~II, Example~7.17.3]{Hartshorne1977}. Before stating the result, 
we need some preliminary considerations.

The right action of~$\pgl(2, \C)$ on~$\p^3$ leaves invariant each of the 
$(d-b)(b+1)$ lines where the map $\Phi_{V_b}$ is not defined, and so 
by the universal property of the blowup (see 
\cite[Corollary~II.7.15]{Hartshorne1977}) it induces an action 
on the blowup~$\bl{B}{\p^3}$ of~$\p^3$ at these lines. In fact\footnote{We 
report here an argument by Daniel Loughran, available at 
\url{https://mathoverflow.net/questions/122922/group-actions-on-blow-ups.}}, 
since $B$ is invariant under the right action of~$\pgl(2, \C)$, then its 
preimage under the map $\pgl(2,\C) \times \p^3 \longrightarrow \p^3$ is 
$\pgl(2,\C) \times B$. Therefore, by the universal property we get a morphism 
$\pgl(2,\C) 
\times \bl{B}{\p^3} \longrightarrow \bl{B}{\p^3}$. By construction, for all $P 
\in \bl{B}{\p^3} \setminus E$, where $E$ is the exceptional divisor, we have
\[
 P \cdot (\sigma \cdot \sigma') = (P \cdot \sigma) \cdot \sigma', 
\qquad P \cdot 
 \left(
 \begin{smallmatrix}
  1 & 0 \\ 0 & 1
 \end{smallmatrix}
 \right)  = P \quad \quad 
 \text{for all } \sigma, \sigma' \in \pgl(2,\C)
\]
hence by continuity these equations hold on the whole~$\bl{B}{\p^3}$, thus 
determining the desired action.

We define a coherent sheaf $\widehat{\mscr{Q}}$ on~$\bl{B}{\p^3}$ as follows. 
Let $\widehat{\tau} \colon \bl{B}{\p^3} \longrightarrow \p^3$ be the blow down 
map. We pull back along $\widehat{\tau}$ the map $\mscr{O}_{\p^3}^{d+1} 
\longrightarrow \mscr{O}_{\p^3}(d)^{b+1}$ defined by the matrix~$\mcal{M}$. 
Then $\widehat{\mscr{Q}}$ is defined as the image of this homomorphism of 
sheaves. It is generated by the global sections $\widehat{\tau}^{\ast}(m_0), 
\dotsc, \widehat{\tau}^{\ast}(m_d)$, where $m_0, \dotsc, m_d$ are the 
columns of the matrix~$\mcal{M}$. The right action of~$\pgl(2,\C)$ 
on~$\bl{B}{\p^3}$ preserves~$\widehat{\mscr{Q}}$. In fact, by 
Lemma~\ref{lemma:right_action} we have that the right action by an element 
$\sigma \in \pgl(2,\C)$ sends each~$m_i$ to a complex linear combination $\sum 
\lambda_{ij} m_j$. Since this holds at every point in $\p^3 \setminus B$, the 
same is true for each $\widehat{\tau}^{\ast}(m_i)$ at every point outside the 
exceptional divisor. Hence, by continuity this must hold on the whole blowup, so 
with this action $\widehat{\mscr{Q}}$ becomes a $\pgl(2,\C)$-equivariant vector 
bundle.

\begin{proposition}
\label{proposition:description_bundle}
With the previously introduced notation, the variety~$\tilde{\p^3}$ can be 
taken to be the blow up of~$\p^3$ at~$B$, the set of $(b + 1)(d - b)$ disjoint 
lines introduced in Proposition~\ref{proposition:bad_lines}, with the 
choice of~$\tau$ as the corresponding blow down morphism. With this choice 
$\mscr{Q}$ is the sheaf associated to the module spanned by the pullbacks 
via~$\tau$ of the generators of~$M$.
\end{proposition}
\begin{proof}
  Define $\widehat{\p^3} = \bl{B}{\p^3}$ and $\widehat{\tau} \colon 
\widehat{\p^3} \longrightarrow \p^3$ to be the corresponding blow down 
morphism. We prove that the sheaf~$\widehat{\mscr{Q}}$ associated to the 
module spanned by the pullbacks of the generators of~$M$ is locally free. This 
implies that we can take $\widehat{\p^3} = \tilde{\p^3}$, $\widehat{\tau} = 
\tau$ and $\widehat{\mscr{Q}} = \mscr{Q}$, so the claim is proved. Since 
$\widehat{\tau}$ is an isomorphism over~$U' = \p^3 \setminus B$, and $\mscr{M}$ 
is locally free on~$U'$ by Proposition~\ref{proposition:bad_lines}, then 
$\widehat{\mscr{Q}}$ is locally free on~$\tau^{-1}(U')$. Hence we only need to 
check that $\widehat{\mscr{Q}}$ is locally free on the exceptional divisors of 
$\widehat{\p^3}$. Since this is a local question, and all the lines in~$\p^3$ 
forming~$B$ are disjoint, we can prove the claim supposing that $B$ is 
constituted by a single line~$L$. By a suitable change of coordinates 
in~$\p^3$ induced by the left action of~$\pgl(2,\C)$, 
we can suppose that $L$ is given by $\{ \gamma = \delta = 0\}$. Hence 
$\bl{L}{\p^3}$ can be written as~$\operatorname{Proj}(\tilde{R})$, with 
$\tilde{R} = \C[\alpha, \beta, \gamma, \overline{\gamma}, \delta, 
\overline{\delta},  w] / (\gamma - \overline{\gamma} w, \delta - 
\overline{\delta} w)$, where we take the $\Z^2$-grading described by the columns 
of the following matrix:
\[
  \bordermatrix{
    & \alpha & \beta & \gamma & \overline{\gamma} & \delta & \overline{\delta} 
    & w \cr
    & 1 & 1 & 1 & 1 & 1 & 1 & 0 \cr
    & 1 & 1 & 1 & 0 & 1 & 0 & 1
  }.
\]
Moreover, one notices that $\tilde{R} \cong \C[\alpha, \beta, 
\overline{\gamma}, \overline{\delta}, w]$ with the previously defined grading. 
The exceptional divisor~$E$ of $\bl{L}{\p^3}$ is then the subvariety $\{ w = 0 
\}$. We pick $P = (\alpha: \beta: \overline{\gamma}: \overline{\delta}: 0) \in 
E$, and we prove that $\widehat{\mscr{Q}}$ is free at~$P$. In order to 
simplify our computations, we consider the right action of~$\pgl(2, \C)$ 
on~$\bl{L}{\p^3}$, which with our choice of coordinates is given as follows: 
if 
$\left( \begin{smallmatrix} u & z \\ y & v 
\end{smallmatrix} \right)$ is an element of~$\pgl(2, \C)$, then its action on 
points of~$\bl{L}{\p^3}$ is given by
\[
 \begin{pmatrix}
  \alpha \\
  \beta \\
  \bar{\gamma} \\
  \bar{\delta} \\
  w
 \end{pmatrix}
 \quad \mapsto \quad
 \begin{pmatrix}
  \alpha u + \beta y \\
  \alpha z + \beta v \\
  \overline{\gamma} u + \overline{\delta} y \\
  \overline{\gamma} z + \overline{\delta} v \\
  w
 \end{pmatrix} .
\]
Therefore, under this action, points in the exceptional divisor 
of~$\bl{L}{\p^3}$ are equivalent to either $(1: 0: 1: 0: 0)$ or $(1: 0: 
0: 1: 0)$, depending on whether $(\alpha:\beta)=(\bar{\gamma}:\bar{\delta})$ 
or $(\alpha:\beta)\neq(\bar{\gamma}:\bar{\delta})$ as points in~$\p^1$. 
Let us consider the case $P = (1: 0: 1: 0: 0)$. Then the image of~$P$ under the 
blow down map, namely $\left( \begin{smallmatrix} 1 & 0 \\ 0 & 0 
\end{smallmatrix} \right)$, belongs to~$B$, and this implies that $(1:0) \in 
\widehat{B}$. Hence by Lemma~\ref{lemma:general_orders} we are in the case 
where $\ord_{(1:0)} G_i = i$ for $i \in \{0, \dotsc, b-1\}$ and $\ord_{(1:0)}G_b 
= b+1$. The module~$\widehat{Q}$ determining $\widehat{\mscr{Q}}$ is generated 
by the columns of the matrix
\[ 
  \widehat{\mcal{Q}}_{ij} \, = \, 
  \frac{\partial^{d} G_{i}(\alpha s + \beta t, \overline{\gamma} w s + 
  \overline{\delta} w t)}{\partial s^{d-j} \partial t^{j}}.
\]
Since we can take $\beta, \overline{\delta}$ and $w$ to be local coordinates 
for~$P$, using simplifications as in the proof of 
Proposition~\ref{proposition:bad_lines} 
we can consider the matrix
\[ 
  \widehat{\mcal{Q}}_{ij} \, = \, 
  \frac{\partial^{d} G_{i}(s, w s + w 
(\overbrace{\overline{\delta}-\beta}^{:=\tilde{D}}) t)}{\partial s^{d-j} 
\partial t^{j}}.
\]
As in Proposition~\ref{proposition:bad_lines}, we can expand
\[
 G_i(s, w s + w \tilde{D} t) \, = \, 
 \sum_{j=0}^{d} \frac{(w \tilde{D} t)^j}{j!} s^{d-j} 
 \frac{\partial^j G_i}{\partial t^j}(1,w).
\]
We can perform the Gaussian elimination that was employed in 
Proposition~\ref{proposition:bad_lines}. The 
difference is that, here, at the $k$-th iteration of the elimination we can 
divide each row from the $(k+1)$-th to the $(b+1)$-th by~$w\tilde{D}$; 
moreover, in this case the Gaussian elimination can be performed only until 
the last-but-one row because of the orders of the polynomials~$G_i$. The 
matrix we obtain has the following shape
\begin{equation}
\label{equation:gaussian_reduced}
 \begin{pmatrix}
  \overline{\mcal{Q}}_{00} & w \tilde{D} \cdot \ast & w^2 \tilde{D}^2 \cdot 
\ast & \cdots \\
  0 & \overline{\mcal{Q}}_{11} & w \tilde{D} \cdot \ast & w^2 \tilde{D}^2 \cdot 
\ast & \cdots \\
  \vdots & 0 & \ddots & \ddots \\
  \vdots & \vdots & & \overline{\mcal{Q}}_{b-1, b-1} & w \tilde{D} \cdot \ast 
& \cdots \\
  0 & 0 & \cdots & 0 &  \overline{\mcal{Q}}_{bb} & w \tilde{D} \cdot \ast 
& \cdots 
 \end{pmatrix} ,
\end{equation}
where the elements $\overline{\mcal{Q}}_{00}, \dotsc, \overline{\mcal{Q}}_{b-1, 
b-1}$ are invertible in the local ring at~$P$, while this is not the case 
for~$\overline{\mcal{Q}}_{bb}$. However, since $\ord_{(1:0)} G_{b} = b+1$, we 
have that $\overline{\mcal{Q}}_{bb} = w \cdot 
\overline{\overline{\mcal{Q}}}_{bb}$ for some invertible element 
$\overline{\overline{\mcal{Q}}}_{bb}$. This implies that the ideal generated by
the last entries of the columns from the $(b+1)$-th to the $(d+1)$-th is 
principal, and so the module spanned by the columns of~$\widehat{\mcal{Q}}$ is 
free at~$P$. 

The case when $P = (1: 0: 0: 1: 0)$ can be treated in an analogous way.
\end{proof}

The following proposition describes the class in the Chow ring of 
the subvariety we are interested in in terms of the Chern classes of the vector 
bundle~$\mscr{Q}$.

\begin{proposition}
\label{proposition:pullback_schubert}
  Let $V_a$ and $V_b$ be general vector subspaces of~$\C[s,t]_{d}$ of 
dimension~$a+1$ and~$b+1$, respectively. Suppose that 
\[
 (a + b - c + 1)(d - c) \, = \, 3.
\]
Let $\mscr{Q}$ be the vector bundle on~$\tilde{\p^3}$ introduced in 
Proposition~\ref{proposition:description_bundle}. Let $\tilde{\Phi}_{V_b} 
\colon \tilde{\p^3} \longrightarrow \Gr{b}{d}$ be the morphism induced 
by~$\mscr{Q}$ and let $Z_{V_a}$ be the Schubert variety in~$\Gr{b}{d}$ as in 
Definition~\ref{definition:schubert_cycle}. Then the class of the
pullback~$\tilde{\Phi}_{V_b}^{*}\bigl([Z_{V_a}]\bigr)$ in the Chow group 
of~$\tilde{\p^3}$ equals
\[
  \begin{array}{lcl}
    \chern{3}{\mscr{Q}} & & 
     \text{if } a + b + 1 - c = 3 \text{ and } d - c = 1, \\
    \chern{3}{\mscr{Q}} - 2\chern{1}{\mscr{Q}}\chern{2}{\mscr{Q}} + 
     \chern{1}{\mscr{Q}}^3 & & 
     \text{if } a + b + 1 - c = 1 \text{ and } d - c = 3.
  \end{array}
\]
\end{proposition}
\begin{proof}
As we mentioned before Proposition~\ref{proposition:description_bundle}, the 
statement follows from the Porteous-Giambelli formula, which states that the 
pullback we are interested in is given by the determinant:
\[
 \left|
 \begin{array}{cccc}
  \chern{(b+1)-(c-a)}{\mscr{Q}} & 
   \chern{(b+2) - (c-a)}{\mscr{Q}} & 
   \cdots & \chern{(b+d-a) - 2(c-a)}{\mscr{Q}} \\
  \chern{b-(c-a)}{\mscr{Q}} & 
   \chern{(b+1) - (c-a)}{\mscr{Q}} & 
   \cdots & \chern{(b+d-a) - 2(c-a)+1}{\mscr{Q}} \\
  \vdots \\
  \chern{b-(d-a)+2}{\mscr{Q}} & 
   \chern{b-(d-a)+3}{\mscr{Q}} & 
   \cdots & \chern{(b+1) - (c-a)}{\mscr{Q}} \\
 \end{array}
 \right|.
\]
In the first case we obtain $\chern{3}{\mscr{Q}}$, while in the second case we 
get
\[
  \left|
  \begin{array}{ccc}
    \chern{1}{\mscr{Q}} & \chern{2}{\mscr{Q}} & \chern{3}{\mscr{Q}} \\
    1 & \chern{1}{\mscr{Q}} & \chern{2}{\mscr{Q}} \\
    0 & 1 & \chern{1}{\mscr{Q}}
  \end{array}
  \right| \, = \, \chern{1}{\mscr{Q}}^3 + \chern{3}{\mscr{Q}} - 
2\chern{1}{\mscr{Q}}\chern{2}{\mscr{Q}}. \qedhere
\]
\end{proof}

Proposition~\ref{proposition:degrees} computes the degrees of the cycles 
obtained in Proposition~\ref{proposition:pullback_schubert} in terms of the 
parameters $a$, $b$, $c$ and $d$ of our initial problem. The result is an 
application of the so-called \emph{Bott residue formula}, which states the 
following. Suppose that $X$ is a smooth variety on which a torus~$\T$ acts, and 
let $\mscr{E}$ be a $\T$-equivariant vector bundle on~$X$ of rank~$r$. Let $p 
\in \C[Z_0, \dotsc, Z_r]$ be a polynomial, and denote by~$p(\mscr{E})$ the 
expression $p\bigl(\chern{0}{\mscr{E}}, \dotsc, \chern{r}{\mscr{E}}\bigr)$. 
Then, the degree of~$p(\mscr{E})$ can be computed by considering the fixed 
locus~$X^{\T}$ of the action of~$\T$ on~$X$, namely
\begin{equation}
\label{equation:bott_formula}
 \deg \bigl(p(\mscr{E})\bigr) = \deg \sum_{L \subseteq X^{\T}} i^{L}_{\ast} 
\left( \frac{p^{\T}(\mscr{E}_{|_L})}{\chernT{d_L}{\mscr{N}_{L/X}}} \right),
\end{equation}
where the sum varies over the components~$L$ of the fixed locus~$X^{\T}$, the 
number~$d_L$ is the codimension of~$L$ in~$X$, the sheaf $\mscr{N}_{L/X}$ is the 
normal bundle of~$L$ in~$X$, each map~$i^{L}$ is the canonical inclusion $L 
\hookrightarrow X$ and the quantity~$p^{\T}(\mscr{E}_{|_L})$ is $p\Bigl( 
\chernT{0}{\mscr{E}_{|_L}}, \dotsc, \chernT{r}{\mscr{E}_{|_L}} \Bigr)$. Here 
$\chernT{i}{\mscr{E}_{|_L}}$ is the so-called \emph{$i$-th $\T$-equivariant 
Chern class of~$\mscr{E}_{|_L}$}, and Equation~\eqref{equation:bott_formula} 
should be read as an equality in the \emph{$\T$-equivariant Chow ring of~$X$}. 
We refer to the lecture notes~\cite{MeirelesAraujo2001}, and to the references 
therein, for the definitions and the properties of these object.  

\begin{proposition}
\label{proposition:degrees}
 With the notation as in Proposition~\ref{proposition:pullback_schubert}, we 
have
\begin{align*}
 \deg \Bigl( \chern{1}{\mscr{Q}}^3 + \chern{3}{\mscr{Q}} - 
2\chern{1}{\mscr{Q}}\chern{2}{\mscr{Q}} \Bigr) &= 6 \tth \binom{a+3}{3} 
\binom{b+3}{3}, \\
 \deg \Bigl( \chern{3}{\mscr{Q}} \Bigr) &= \frac{1}{6} \tth ab (a^2 - 1)(b^2 - 
1).
\end{align*}
\end{proposition}
\begin{proof}
As we described before Proposition~\ref{proposition:description_bundle}, the 
natural right action of~$\pgl(2,\C)$ on~$\p^3$ determines an action 
of~$\pgl(2,\C)$ on the blowup~$\tilde{\p^{3}}$ and on the sheaf~$\mscr{Q}$ such 
that $\mscr{Q}$ is $\pgl(2,\C)$-equivariant.
 In particular, we have an action on~$\tilde{\p^3}$ by the torus $\T \cong 
\bigl( \C^{\ast} \bigr)^2$ of $2 \times 2$ invertible diagonal matrices. This 
action determines an action on the vector bundle~$\mscr{Q}$, making it into a 
$\T$-equivariant vector bundle.
 Hence we are in the situation of Bott residue formula. First of all, we 
compute the fixed locus of the action of~$\T$ on~$\tilde{\p^3}$.

\begin{lemma}
\label{lemma:fixed_locus}
 The fixed locus of the action of~$\T$ on~$\tilde{\p^3}$ is constituted of two 
lines and $2(b+1)(d-b)$ points.
\end{lemma}
\begin{proof}
We start with the computation of the fixed locus of the action of~$\T$ 
on~$\p^3$. Here, a direct computation shows that this locus is constituted of 
the union of the two lines $L_1 = \{ \alpha = \gamma = 0\}$ and $L_2 = \{\beta 
= \delta = 0\}$. Then by construction the fixed locus of the action of~$\T$ on 
$\tilde{\p^3}$ is contained in the preimage $\tau^{-1}(L_1 \cup L_2)$, where 
$\tau \colon \tilde{\p^3} \longrightarrow \p^3$ is the canonical map. By 
continuity, the strict transforms of~$L_1$ and~$L_2$, which we 
denote by~$\tilde{L}_1$ and~$\tilde{L}_2$, are fixed by~$\T$. In order to 
determine whether some other points in $\tau^{-1}(L_i)$ for $i \in \{1,2\}$ are 
fixed by~$\T$, we can argue as in 
Proposition~\ref{proposition:description_bundle} and do the computations 
assuming that $\tilde{\p^3}$ is the blowup of~$\p^3$ along the line $\{ \gamma 
= \delta = 0\}$. With the same choice of coordinates as in 
Proposition~\ref{proposition:description_bundle}, the action of an element 
$\left( \begin{smallmatrix} u & 0 \\ 0 & v \end{smallmatrix} \right)$ on a 
point $(\alpha: \beta: \bar{\gamma}: \bar{\delta}: w)$ is given by:
\[
 \begin{pmatrix}
  \alpha \\
  \beta \\
  \bar{\gamma} \\
  \bar{\delta} \\
  w
 \end{pmatrix}
 \quad \mapsto \quad
 \begin{pmatrix}
  \alpha u \\
  \beta v \\
  \overline{\gamma} u \\
  \overline{\delta} v \\
  w
 \end{pmatrix} \,.
\]
Hence a direct computation shows that the fixed components are:
\begin{align*}
 \alpha = \bar{\gamma} = 0, & & \beta = \bar{\delta} = 0, \\
 w = \alpha = \bar{\delta} = 0, & & w = \beta = \bar{\gamma} = 0.
\end{align*}
Notice that the first two are the strict transforms of~$L_1$ 
and $L_2$, while the second two are two isolated points on the exceptional 
divisor.
\end{proof}

Once the fixed locus of the torus action is computed, we know from the general 
theory (see \cite[Section 3.5, Theorem 32]{Luke1998}) that the restriction of a 
$\T$-equivariant vector bundle to the fixed locus of the $\T$-action on the 
base splits as a direct sum of eigenbundles. It is crucial to compute this 
decomposition in order to determine the $\T$-equivariant Chern classes of the 
vector bundle. In fact, suppose that $\mscr{E}$ is a $\T$-equivariant vector 
bundle on a smooth $\T$-variety $X$, and the action of~$\T$ on~$X$ is trivial, 
then we have a decomposition $\mscr{E} = \bigoplus \mscr{E}_{\chi}$ into 
eigenbundles, where $\chi$ varies over the characters of~$\T$. In this 
situation, we can express the $\T$-equivariant Chern classes of each 
eigenbundle $\mscr{E}_{\chi}$ as (see \cite[Lemma~3]{Edidin1998})
\begin{equation}
\label{equation:compute_equivariant}
 \chernT{i}{\mscr{E}_{\chi}} = 
 \sum_{j \leq i} \binom{r - j}{i - j} \chern{j}{\mscr{E}_{\chi}} \chi^{i - j},
\end{equation}
where $r$ is the rank of~$\mscr{E}_{\chi}$.
At this point Whitney sum formula, which holds also in the equivariant setting, 
provides the equivariant Chern classes of~$\mscr{E}$.

We therefore proceed by computing the decomposition into eigensubbundles of the 
restriction of the vector bundle~$\mscr{Q}$ to the various components of the 
fixed locus of the action of~$\T$ on $\tilde{\p^3}$.

\begin{lemma}
\label{lemma:eigenbundles_lines}
 The restriction of~$\mscr{Q}$ to~$\tilde{L}_i$ for $i \in \{1,2\}$ splits into 
$b+1$ eigenbundles of rank~$1$, each isomorphic to $\mscr{O}_{\p^1}$. The 
corresponding characters are $u^d, \dotsc, u^{d-b} v^b$ in the case 
of~$\tilde{L}_2$ and $u^b v^{d-b}, \dotsc, v^d$ in the case of~$\tilde{L}_1$.
\end{lemma}
\begin{proof}
 If $m_0, \dotsc, m_d$ are the generators of the module~$M$, whose 
sheafification is the coherent sheaf~$\mscr{M}$ on~$\p^3$, then an element 
$\left( \begin{smallmatrix} u & 0 \\ 0 & v \end{smallmatrix} \right)$ sends 
each~$m_j$ to~$(u^{d-j} v^{j}) m_i$. This follows from the proof of 
Lemma~\ref{lemma:right_action}. Let $\tilde{m}_0, \dotsc, \tilde{m}_d$ be the 
generators of~$Q$, the module whose sheafification is~$\mscr{Q}$, corresponding 
to $m_0, \dotsc, m_d$. Then $\left( \begin{smallmatrix} u & 0 \\ 0 & v 
\end{smallmatrix} \right)$ sends~$\tilde{m}_j$ to~$(u^{d-j} v^{j}) \tilde{m}_j$. 
The restriction~$\mscr{Q}_{|_{\tilde{L}_1}}$ of~$\mscr{Q}$ to~$\tilde{L}_i$ is 
generated by the restrictions of the elements~$\tilde{m}_0, \dotsc, 
\tilde{m}_d$, and by what we have just proved these restrictions generate 
eigensubbundles of mutually different characters; each of the eigensubbundles is 
isomorphic to $\mscr{O}_{\p^1}$ because it is generated by a single section of 
degree~$0$. To conclude the proof, we only have to compute which of the 
restrictions of the elements $\tilde{m}_0, \dotsc, \tilde{m}_d$ generate 
$\mscr{Q}_{|_{\tilde{L}_i}}$, and this can be checked by looking at the stalk of 
the vector bundle at an arbitrary point of each line. 

If we pick a point $P \in \tilde{L}_2$ that does not lie on the exceptional 
divisor, then we can do our computations in~$\p^3$. So, using the left action 
of~$\pgl(2,\C)$ on~$\p^3$, we can suppose that $P = \left( 
\begin{smallmatrix} 1 & 0 \\ 0 & 0 \end{smallmatrix} \right)$, namely $\beta = 
\gamma = \delta = 0$. Hence we are in the situation of the proof of 
Proposition~\ref{proposition:bad_lines}, and here we see that the restriction of 
the first $b+1$ elements $\tilde{m}_0, \dotsc, \tilde{m}_b$ generate 
$\mscr{Q}_{|_{\tilde{L}_2}}$. Similarly, when $P \in \tilde{L}_1$ does not lie 
on the exceptional divisor, an analogous version of the local analysis 
performed in Proposition~\ref{proposition:bad_lines} shows that 
$\mscr{Q}_{|_{\tilde{L}_1}}$ is generated by $\tilde{m}_{d-b}, \dotsc, 
\tilde{m}_d$. In this case one can perform Gaussian elimination from 
the right to the left because the situation is ``mirrored'' with respect to the 
previous one.
\end{proof}

\begin{lemma}
\label{lemma:eigenbundles_points}
 The restriction of~$\mscr{Q}$ to each of the $2(b+1)(d-b)$ fixed points splits 
into $b+1$ eigenbundles of rank~$1$, each hence isomorphic to~$\C$. The 
corresponding characters are $u^d, \dotsc, u^{d-(b-1)} v^{b-1}, 
u^{d-(b+1)}v^{b+1}$ in the case 
of the points above~$L_2$ and $u^{b+1}v^{d-(b+1)}, u^{b-1} v^{d-(b-1)}, \dotsc, 
v^d$ in the case of the points above~$L_1$.
\end{lemma}
\begin{proof}
 As in the proof of Lemma~\ref{lemma:eigenbundles_lines}, we know that each
section~$\tilde{m}_j$ of~$\tilde{\mscr{Q}}$ is sent to~$(u^{d-j} v^{j}) 
\tilde{m}_j$ by the action of the torus~$\T$. Hence, as in 
Lemma~\ref{lemma:eigenbundles_lines}, 
the restriction~$\tilde{\mscr{Q}}_{|_{P}}$, where $P$ is any of the 
$2(b+1)(d-b)$ fixed points, splits into the direct sum of trivial 
eigensubbundles; to determine the relevant characters it is enough to understand 
which of the restrictions to~$P$ of the sections~$\{ \tilde{m}_j \}$ do not 
vanish.

Suppose that $P$ is a point over~$L_2$. Since what we need to perform is a 
local computation, we can put ourselves in the situation of 
Proposition~\ref{lemma:fixed_locus}. Then, we can take $w$, $\beta$ and 
$\bar{\gamma}$ as local coordinates for~$P$, setting $\alpha = \bar{\delta} = 
1$. In this situation the matrix~$\mcal{Q}$ whose columns generate $\mscr{Q}$ 
at~$P$ has entries
\[
 \frac{\partial^{d}G_i(s + \beta t, \bar{\gamma}ws + wt)}{\partial s^{d-j}t^j}.
\]
By employing the substitution $s \mapsto s - \beta t$ and setting $\tilde{D} = 
\bar{\gamma}\beta - 1$, the entries of the matrix become
\[
 \frac{\partial^{d}G_i(s, \bar{\gamma}ws + w\tilde{D}t)}{\partial s^{d-j}t^j}.
\]
Using the Taylor expansion already employed in 
Proposition~\ref{proposition:description_bundle}, we obtain
\[
 \mcal{Q}_{ij} = 
 \frac{(w \tilde{D})^{j}}{j!} \cdot \frac{\partial^j G_i}{\partial t^{j}} (1, 
\bar{\gamma}w).
\]
Now we can proceed with the Gaussian elimination as described in 
Proposition~\ref{proposition:description_bundle} until we reach the situation 
of Equation~\eqref{equation:gaussian_reduced}. From the shape of the matrix we 
infer that the first $b$ columns are linearly independent, and in order to 
prove our claim we just have to show that the $b+2$-th column gives a system of 
free generators for~$\mscr{Q}$ at~$P$. The last row of the matrix~$\mcal{Q}$ 
has the following shape:
\[
 \bigl( \, \underbrace{0 \quad \dotsc \quad 0}_{b \text{ zeros}} \quad 
\overline{\mcal{Q}}_{bb} \quad w \tilde{D} \cdot \overline{\mcal{Q}}_{b,b+1} 
\quad  w^2\tilde{D}^2 \cdot \ast \quad \cdots \, \bigr).
\]
Here, as in Proposition~\ref{proposition:description_bundle}, 
the polynomial~$\overline{\mcal{Q}}_{bb}$ is of the form $H_b(1, 
\bar{\gamma}w)$ for some polynomial $H_b(x,y)$ such that $\ord_{(1:0)}H_b = 1$, 
while $\overline{\mcal{Q}}_{b,b+1}$ is of the form $H_{b+1}(1, \bar{\gamma}w)$, 
where $\ord_{(1:0)}H_{b+1} = 0$. This implies that $\overline{\mcal{Q}}_{bb} = 
\bar{\gamma}w + \ldots$, while $\overline{\mcal{Q}}_{b,b+1}$ is invertible in 
the local ring at~$P$. Since also $\tilde{D}$ is invertible in that ring, it 
follows that the ideal generated by the entries of the last row of~$\mcal{Q}$ 
is $(w)$, and it is generated by $w\tilde{D} \cdot 
\overline{\mcal{Q}}_{b,b+1}$. This implies that the sections that do not vanish 
at~$P$ are $\tilde{m}_0, \dotsc, \tilde{m}_{b-1}$ and $\tilde{m}_{b+1}$, 
showing our claim.

The case of points over $L_1$ is similar, as discussed in 
Lemma~\ref{lemma:eigenbundles_lines}.
\end{proof}

\begin{lemma}
\label{lemma:equivariant_chern_lines}
 For $k \in \{1,2,3\}$ let $\varsigma_k(y_0, \dotsc, y_t)$ be the 
$k$-th elementary symmetric polynomial in the variables $y_0, \dotsc, y_t$.
 The $k$-th equivariant Chern classes of the restriction of~$\mscr{Q}$ 
to~$\tilde{L}_i$ for $i \in \{1,2\}$ are 
\begin{gather*}
 \chernT{k}{\mscr{Q}_{|_{\tilde{L}_1}}} = \varsigma_k \bigl((d-i)v + i \, u, 
\text{ for } i \in \{0, \dotsc, b\} \bigr), \\
 \chernT{k}{\mscr{Q}_{|_{\tilde{L}_2}}} = \varsigma_k \bigl((d-i)u + i \, v, 
\text{ for } i \in \{0, \dotsc, b\} \bigr).
\end{gather*}
\end{lemma}
\begin{proof}
 For each eigensubbundle~$\mscr{E}_{\chi}$ of~$\mscr{Q}_{|_{\tilde{L}_i}}$ of 
character~$\chi$ we have $\chern{0}{\mscr{E}_{\chi}} = 1$ and 
$\chern{i}{\mscr{E}_{\chi}} = 0$ for all $i \geq 1$. In fact, by 
Lemma~\ref{lemma:eigenbundles_lines} every subbundle~$\mscr{E}_{\chi}$ is
isomorphic to~$\mscr{O}_{\p^1}$. Hence, by 
Equation~\eqref{equation:compute_equivariant} we have 
$\chernT{1}{\mscr{E}_{\chi}} = \chi$  and 
$\chernT{i}{\mscr{E}} = 0$ for $i \geq 2$. For example, if we consider 
~$\mscr{Q}_{|_{\tilde{L}_2}}$ and we take $\chi = u^{d-i}v^{i}$, namely 
$\mscr{E}_{\chi}$ is the eigensubbundle generated by the restriction 
to~$\tilde{L}_2$ of the global section~$\tilde{m}_i$, then we have 
$\chernT{1}{\mscr{E}_{\chi}} = (d-i)u + i \, v$ (here the character is reported 
in logarithmic notation). The statement then follows from the 
Whitney sum formula for the Chern class of a direct sum and the description of 
the characters of the eigensubbundles provided 
by Lemma~\ref{lemma:eigenbundles_lines}.
\end{proof}

\begin{lemma}
\label{lemma:equivariant_chern_points}
 For $k \in \{1,2,3\}$ let $\varsigma_k(y_0, \dotsc, y_t)$ be the 
$k$-th elementary symmetric polynomial in the variables $y_0, \dotsc, y_t$.
 The $k$-th equivariant Chern classes of the restriction of~$\mscr{Q}$ 
to each of the $2(b+1)(d-b)$ fixed points are
\[
 \chernT{k}{\mscr{Q}_{P}} = \varsigma_k 
 \Biggl(
  \begin{array}{c}
   (d-i)v + i \, u, \\ 
   i \in \{0, \dotsc, b-1\}
  \end{array} , 
  \bigl(d-(b+1)\bigr)v + (b+1)u 
 \Biggr)
\]
in the case of the points~$P$ over~$L_1$, and 
\[
 \chernT{k}{\mscr{Q}_{P}} = \varsigma_k 
 \Biggl(
  \begin{array}{c}
   (d-i)u + i \, v, \\ 
   i \in \{0, \dotsc, b-1\}
  \end{array} , 
  \bigl(d-(b+1)\bigr)u + (b+1)v 
 \Biggr)
\]
in the case of the points~$P$ over~$L_2$.
\end{lemma}
\begin{proof}
 The proof is analogous to the one of Lemma~\ref{lemma:equivariant_chern_lines}.
\end{proof}

\begin{lemma}
\label{lemma:normal_lines}
 The normal bundle of~$\tilde{L}_1$ (respectively, $\tilde{L}_2$) 
in~$\tilde{\p^3}$ is an eigenbundle of character~$u/v$ (respectively, $v/u$). 
The Chern polynomial of $\mscr{N}_{\tilde{L}_i/\tilde{\p^3}}$ is $1 + 
\mbox{$\bigl(2 - (b+1)(d-b)\bigr)h$}$, where $h$ is the class of a point, for 
both $i = 1$ and $i = 2$.
\end{lemma}
\begin{proof}
We prove the statement for~$\tilde{L}_1$, the argument for~$\tilde{L}_2$ is 
analogous. We start by showing that the whole bundle 
$\mscr{N}_{\tilde{L}_1/\tilde{\p^3}}$ is an eigenbundle of character $u/v$. It 
suffices to show this locally at a point,
since the decomposition into eigensubbundles is canonical. Consider hence a 
point $P = (0:\beta:0:\bar{\delta}:w)$ on~$\tilde{L}_1$ (here we use the 
notation as in Lemma~\ref{lemma:fixed_locus}). We can pick affine coordinates 
$\alpha$, $\bar{\gamma}$, and~$w$ for~$P$, setting $\beta = \bar{\delta} = 1$. 
In these local coordinates, the action of~$\T$ is given by
\begin{equation}
\label{equation:action_affine1}
 \begin{pmatrix}
  \alpha \\
  \bar{\gamma} \\
  w
 \end{pmatrix}
 \quad \mapsto \quad
 \begin{pmatrix}
  \nicefrac{u}{v} \, \alpha \\
  \nicefrac{u}{v} \, \bar{\delta} \\
  w
 \end{pmatrix}.
\end{equation}
In fact, recall that the action of~$\T$ on the coordinates $\alpha, \beta, 
\bar{\gamma}, \bar{\delta}, w$ is
\[
 \begin{pmatrix}
  \alpha \\
  1 \\
  \bar{\gamma} \\
  1 \\
  w
 \end{pmatrix}
 \quad \mapsto \quad
 \begin{pmatrix}
  u \, \alpha \\
  v \\
  u \, \bar{\gamma} \\
  v \\
  w
 \end{pmatrix} = 
 \begin{pmatrix}
  \nicefrac{u}{v} \, \alpha \\
  1 \\
  \nicefrac{u}{v} \, \bar{\gamma} \\
  1 \\
  w
 \end{pmatrix} ,
\]
where we use the special grading of the coordinates for the equality on the 
right. Because of the choice of coordinates, and recalling that $\tilde{L}_1$ 
is defined by $\alpha = \bar{\gamma} = 0$, we can write 
\begin{equation}
\label{equation:normal_bundle}
 \left( \mscr{N}_{\tilde{L}_1/\tilde{\p^3}} \right)_{\!P} \cong 
\frac{\left\langle \partial_{\alpha}, \partial_{\bar{\gamma}}, 
\partial_{w} \right\rangle}{\left\langle \partial_{w} \right\rangle}.
\end{equation}
Since the action on the tangent space at~$P$ of~$\tilde{\p^3}$ is given by the 
Jacobian of Equation~\eqref{equation:action_affine1}, and because of the 
description of the normal bundle in Equation~\eqref{equation:normal_bundle} we 
see that the action on the normal bundle is given by the first principal $2 
\times 2$ minor of that Jacobian, which is a diagonal matrix with~$u/v$ as 
diagonal entries. Hence the whole~$\mscr{N}_{\tilde{L}_1/\tilde{\p^3}}$ is an 
eigenbundle of character~$u/v$.

Let us now compute the Chern polynomial of the normal bundle of~$\tilde{L}_1$; 
the same proof works for~$\tilde{L}_2$. We write~$L$ for~$L_1$. 
Since both~$L$ and~$B$ (the blowup center) are regularly embedded in~$\p^3$, 
and the intersection $L \cap B$ is regularly embedded in both~$L$ and~$B$, 
we can use the result of Aluffi \cite[Section 4.3]{Aluffi2010} to compute the 
Chern polynomial~$c(\mscr{N}_{\tilde{L} / \tilde{\p^3}})$. In fact, both $L$ 
and $B$ are 
contained in the smooth quadric $Q = \{ \alpha \delta - \beta \gamma = 0\}$, 
and inside~$Q$ they intersect properly, since their intersection is 
equidimensional of codimension~$2$ in~$Q$. If $\tilde{Q}$ is the strict 
transform of~$Q$ in~$\tilde{\p^3}$, and $\tau \colon \tilde{\p^3} 
\longrightarrow \p^3$ denotes the blow down map, and $i_{\tilde{L}} \colon 
\tilde{L} \hookrightarrow \tilde{Q}$ and $i_{\tilde{Q}} \colon \tilde{Q} 
\hookrightarrow 
\tilde{\p^3}$ denote the closed immersions, then
\[
 c(\mscr{N}_{\tilde{L} / \tilde{\p^3}}) =
 \tau^{\ast}_{|_{\tilde{L}}} c(\mscr{N}_{L / Q}) \cdot i_{\tilde{L}}^{\ast} 
 c \bigl(
  \tau^{\ast}_{|_{\tilde{Q}}} \mscr{N}_{Q / \p^3} \otimes 
  i_{\tilde{Q}}^{\ast} \mscr{O}_{\tilde{\p^3}}(-E) 
 \bigr).
\]
The formula before can be written as
\[
 c(\mscr{N}_{\tilde{L} / \tilde{\p^3}}) =
 c(\mscr{N}_{L / Q}) \cdot 
 c \bigl(\mscr{N}_{Q / \p^3} \otimes \mscr{O}_{\tilde{\p^3}}(-E) \bigr)
\]
if we omit the pullbacks. Since by \cite[Example V.1.4.1]{Hartshorne1977} the 
degree of $\mscr{N}_{L / Q}$ equals the self-intersection of~$L$ inside~$Q$, we 
have $\deg (\mscr{N}_{L / Q}) = 0$ and so $c(\mscr{N}_{L / Q})$ equals~$1$. By 
\cite[Adjunction Formula 
I]{Griffiths1978} we have $\mscr{N}_{Q / \p^3} \cong \mscr{O}_{\p^3}(Q)_{|Q}$ 
and so, as a sheaf on~$\p^1 \times \p^1$, the latter is $\mscr{O}_{\p^1 \times 
\p^1}(2,2)$. Since $\tau_{|_{\tilde{Q}}}$ is an isomorphism (in fact the blowup 
center~$B$ is a divisor in~$Q$), we can compute $\tau^{\ast}_{|_{\tilde{Q}}} 
\mscr{N}_{Q / \p^3} \otimes i_{\tilde{Q}}^{\ast} \mscr{O}_{\tilde{\p^3}}(-E)$ 
as a sheaf on~$\p^1 \times \p^1$; this amounts to shift $\mscr{N}_{Q / 
\p^3}$ by $\mscr{O}_{Q}(-B)$, obtaining $\mscr{O}_{\p^1 \times \p^1}(2,2-K)$, 
where $K = (b+1)(d-b)$. Notice that here we used that the class in~$\p^1 \times 
\p^1$ of the lines in~$B$ is~$(0,1)$, thus the class of~$L$ is $(1,0)$. 
Eventually, restricting this bundle to~$\tilde{L}$ gives 
$\mscr{O}_{\tilde{L}}(2-K)$, because the intersection product in~$\p^1 \times 
\p^1$ of the classes $(1,0)$ and $(2,2-K)$ equals $2-K$. Hence we obtain
\[
 c(\mscr{N}_{\tilde{L} / \tilde{\p^3}}) =
 1 + \bigl(2 - (b+1)(d-b)\bigr)h,
\]
where $h$ is the class of a point.
\end{proof}

\begin{lemma}
\label{lemma:normal_points}
 The normal bundle --- namely, the tangent space --- of each of the 
$2(b+1)(d-b)$ fixed points splits as the sum of two eigenbundles as follows:
\begin{itemize}
 \item 
 one of rank~$1$ of character $v/u$ and another of rank~$2$ of character~$u/v$ 
for the points over $L_{1}$;
 \item
 one of rank~$1$ of character $u/v$ and another of rank~$2$ of character~$v/u$ 
for the points over $L_{2}$.
\end{itemize}
\end{lemma}
\begin{proof}
 Let $P$ be a point over $L_1$. We compute the action of the torus locally 
around~$P$. Thus we can suppose that we are in the setting of 
Lemma~\ref{lemma:fixed_locus}, namely $P$ has equations $w = \alpha = 
\bar{\delta} = 0$. Similarly as what we did in Lemma~\ref{lemma:normal_lines}, 
we can pick affine coordinates $w, \alpha, \bar{\delta}$ for $P$, setting 
$\beta = \bar{\gamma} = 1$. In these local coordinates, the action of~$\T$ is 
given by
\begin{equation}
\label{equation:action_affine2}
 \begin{pmatrix}
  \alpha \\
  \bar{\delta} \\
  w
 \end{pmatrix}
 \quad \mapsto \quad
 \begin{pmatrix}
  \nicefrac{u}{v} \, \alpha \\
  \nicefrac{v}{u} \, \bar{\delta} \\
  \nicefrac{u}{v} \, w
 \end{pmatrix}
\end{equation}
because the action of~$\T$ on the coordinates $\alpha, \beta, 
\bar{\gamma}, \bar{\delta}, w$ is
\[
 \begin{pmatrix}
  \alpha \\
  1 \\
  1 \\
  \bar{\delta} \\
  w
 \end{pmatrix}
 \quad \mapsto \quad
 \begin{pmatrix}
  u \, \alpha \\
  v \\
  u \\
  v \, \bar{\delta} \\
  w
 \end{pmatrix} = 
 \begin{pmatrix}
  \alpha \\
  \nicefrac{v}{u} \\
  1 \\
  \nicefrac{v}{u} \, \bar{\delta} \\
  w
 \end{pmatrix} =
 \begin{pmatrix}
  \nicefrac{u}{v} \, \alpha \\
  1 \\
  1 \\
  \nicefrac{v}{u} \, \bar{\delta} \\
  \nicefrac{u}{v} \, w
 \end{pmatrix},
\]
where we use the special grading of the coordinates for the equalities on the 
right.
Since the action on the tangent bundle is given by the Jacobian of the 
action in Equation~\eqref{equation:action_affine2}, we obtain two 
eigensubbundles of the desired rank and character.

The argument when $P$ is a point over $L_2$ is identical.
\end{proof}

\begin{lemma}
\label{lemma:equivariant_normal_lines}
 The second equivariant Chern class of the normal bundle of~$\tilde{L}_i$ is
\begin{gather*}
 \chernT{2}{\mscr{N}_{\tilde{L}_1/\tilde{\p^3}}} = 
  (u-v)^2+(u-v)\bigl(2-(b+1)(d-b)\bigr)h, \\
 \chernT{2}{\mscr{N}_{\tilde{L}_2/\tilde{\p^3}}} = 
  (v-u)^2+(v-u)\bigl(2-(b+1)(d-b)\bigr)h.
\end{gather*}
\end{lemma}
\begin{proof}
 We argue as in Lemma~\ref{lemma:equivariant_chern_lines} using the results of 
Lemma~\ref{lemma:normal_lines}.
\end{proof}

\begin{lemma}
\label{lemma:equivariant_normal_points}
 The third equivariant Chern class of the normal bundle of each 
of the $2(b+1)(d-b)$ fixed points is
\begin{align*}
 \chernT{3}{\mscr{N}_{P/\tilde{\p^3}}} = 
  (v-u)^3 \qquad & \text{for points\ } P \text{ above } L_1 \text{ and} \\
 \chernT{3}{\mscr{N}_{P/\tilde{\p^3}}} = 
  (u-v)^3 \qquad & \text{for points\ } P \text{ above } L_2.
\end{align*}
\end{lemma}
\begin{proof}
 Suppose that $P$ is a point over $L_2$, then by Lemma~\ref{lemma:normal_points} 
we have that $\mscr{N}_{P/\tilde{\p^3}} \cong E_1 \oplus E_2$ where $E_i$ are 
vector spaces of dimension $i$ and of characters~$u/v$ and~$v/u$, 
respectively. Using \cite[Lemma~3]{Edidin1998} we find that 
$\chernT{1}{E_1} = u - v$, while $\chernT{i}{E_1} = 0$ for all $i 
\geq 2$, and $\chernT{2}{E_2} = (v-u)^2$, while $\chernT{i}{E_1} = 0$ for all 
$i \geq 3$. Hence $\chernT{3}{\mscr{N}_{P/\tilde{\p^3}}} = \chernT{1}{E_1} 
\chernT{2}{E_2} = (u-v)^3$ by the Whitney formula. 

The proof for points over $L_1$ is analogous.
\end{proof}
We can now finally use Bott residue formula to prove 
Proposition~\ref{proposition:degrees}. We only present the 
computation of $\deg \bigl( \chern{3}{\mscr{Q}} \bigr)$: the other 
case is analogous, but the computations are more tedious. 
Equation~\eqref{equation:bott_formula} gives
\begin{multline*}
 \deg \bigl(\chern{3}{\mscr{Q}} \bigr) = 
  \deg \, i^{\tilde{L}_1}_{\ast} \left( 
\frac{\chernT{3}{\mscr{Q}_{|_{\tilde{L}_1}}}}{\chernT{2}{\mscr{N}_{\tilde{L 
}_1/\tilde{\p^3}} } } \right) + \deg \, i^{\tilde{L}_2}_{\ast} \left(
\frac{\chernT{3}{\mscr{Q}_{|_{\tilde{L}_2}}}}{\chernT{2}{\mscr{N}_{\tilde{L}
_2/\tilde{\p^3}}}} \right) \\
+ (b+1)(d-b) \underbrace{\deg \, i^{P_1}_{\ast} \left(
\frac{\chernT{3}{\mscr{Q}_{|_{P_1}}}}{\chernT{3}{\mscr{N}_{P_1/\tilde{\p^3}}}}
\right)}_{\text{a point over } L_1}
+ (b+1)(d-b) \underbrace{\deg \, i^{P_2}_{\ast} \left(
\frac{\chernT{3}{\mscr{Q}_{|_{P_2}}}}{\chernT{3}{\mscr{N}_{P_2/\tilde{\p^3}}}}
\right)}_{\text{a point over } L_2}.
\end{multline*}
The previous results allow us to compute each of the four summands:
\begin{align*}
\deg \, & i^{\tilde{L}_1}_{\ast} \left( 
\frac{\chernT{3}{\mscr{Q}_{|_{\tilde{L}_1}}}}{\chernT{2}{\mscr{N}_{\tilde{L 
}_1/\tilde{\p^3}} } } \right) = \deg 
\frac{\varsigma_3 \bigl((d-i)v + i \, u, \, 
i \in \{0, \dotsc, b\} \bigr)}{(u-v)^2+(u-v)\bigl(2-(b+1)(d-b)\bigr)h} \\
&= \frac{\varsigma_3 \bigl((d-i)v + i \, u, \, 
i \in \{0, \dotsc, b\} \bigr)}{u-v} \deg \frac{(u-v) - 
\bigl(2-(b+1)(d-b)\bigr)h}{(u-v)^2} \\
&= \frac{\varsigma_3 \bigl((d-i)v + i \, u, \, 
i \in \{0, \dotsc, b\} \bigr) \bigl(-2 + (b+1)(d-b)\bigr)}{(u-v)^3} ,
\end{align*}
(where in the second equality we multiplied both numerator and denominator by 
$(u-v) - \bigl(2-(b+1)(d-b)\bigr)h$, and we used the 
fact that $h^2 = 0$ in the Chow ring of~$\tilde{L}_1$)
\begin{align*}
\deg \, i^{\tilde{L}_2}_{\ast} \left( 
\frac{\chernT{3}{\mscr{Q}_{|_{\tilde{L}_2}}}}{\chernT{2}{\mscr{N}_{\tilde{L 
}_2/\tilde{\p^3}} } } \right) &=  \frac{\varsigma_3 \bigl((d-i)u + i \, v, \, 
i \in \{0, \dotsc, b\} \bigr) \bigl(-2 + (b+1)(d-b)\bigr)}{(v-u)^3} , \\
\deg \, i^{P_1}_{\ast} \left(
\frac{\chernT{3}{\mscr{Q}_{|_{P_1}}}}{\chernT{3}{\mscr{N}_{P_1/\tilde{\p^3}}}}
\right) &= \frac{1}{(v-u)^3} \, \varsigma_3 \Biggl(\!\!\!
\begin{array}{c}
(d-i)v + i \, u, \\ 
i \in \{0, \dotsc, b-1\}
\end{array} , 
\bigl(d-(b+1)\bigr)v + (b+1)u \! \Biggr) , \\
\deg \, i^{P_2}_{\ast} \left(
\frac{\chernT{3}{\mscr{Q}_{|_{P_2}}}}{\chernT{3}{\mscr{N}_{P_2/\tilde{\p^3}}}}
\right) &= \frac{1}{(u-v)^3} \, \varsigma_3 \Biggl(\!\!\!
\begin{array}{c}
(d-i)u + i \, v, \\ 
i \in \{0, \dotsc, b-1\}
\end{array} , 
\bigl(d-(b+1)\bigr)u + (b+1)v \! \Biggr) .
\end{align*}
Hence in this way we have expressed $\deg \bigl(\chern{3}{\mscr{Q}} \bigr)$ as 
a rational function in~$u$ and~$v$. Therefore this rational function is 
constant, and so we can assign arbitrary values to~$u$ and~$v$ (as long as 
fractions have non-zero denominator), and we will obtain the same number. We 
take $u = 1$ and $v = 0$, and we write $K$ for the quantity $(b+1)(d-b)$. In 
this way, we get
\begin{align*}
  \deg \bigl(\chern{3}{\mscr{Q}} \bigr) = 
  &\phantom{+\ }(K-2) \,\varsigma_3
   \bigl(i, \text{ for } i \in \{0, \dotsc, b\}\bigr) \\
  &-(K-2) \,\varsigma_3\bigl(d-i, \text{ for } i \in \{0, \dotsc, b\}\bigr) \\ 
  &-K \,\varsigma_3 \bigl(i, \text{ for } i \in \{0, \dotsc, b-1\},(b+1)\bigr)\\
  &+ K \, \varsigma_3 
   \bigl(d-i, \text{ for } i \in \{0, \dotsc, b-1\}, \bigl(d-(b+1)\bigr) \bigr).
\end{align*}
The previous formula can be simplified further, noticing that 
\begin{multline*}
 \varsigma_3 \bigl(i, \text{ for } i \in \{0, \dotsc, b-1\},(b+1)\bigr) = \\
 \varsigma_3 \bigl(i, \text{ for } i \in \{0, \dotsc, b\}\bigr) + 
 \varsigma_2 \bigl(i, \text{ for } i \in \{0, \dotsc, b-1\}\bigr).
\end{multline*}
We get
\begin{align*}
 \deg \bigl(\chern{3}{\mscr{Q}} \bigr) = 
 &-2 \,\varsigma_3\bigl(i, \text{ for } i \in \{0, \dotsc, b\}\bigr) 
\\
 &+2 \,\varsigma_3\bigl(d-i, \text{ for } i \in \{0, \dotsc, b\}\bigr) \\ 
 &-K \,\varsigma_2 \bigl(i, \text{ for } i \in \{0, \dotsc, b-1\}\bigr) \\
 &-K \,\varsigma_2 \bigl(d-i, \text{ for } i \in \{0, \dotsc, b-1\}\bigr).
\end{align*}
If we write
\[
 \varsigma_3\bigl(i, \text{ for } i \in \{0, \dotsc, b\}\bigr) = 
 \sum_{h=0}^{b} \; \sum_{l = h+1}^{b} \; \sum_{m = l+1}^{b} h l m,
\]
and similarly for the other summands, then we obtain the statement using 
standard techniques in summation or a symbolic summation software as, for 
example, Mathematica. We thank Christoph Koutschan for helping us with this 
symbolic summation problem.
\end{proof}

Proposition~\ref{proposition:pullback_schubert} and~\ref{proposition:degrees} imply:

\begin{theorem}
\label{theorem:algebraic}
  Let $V_a$ and $V_b$ be general vector subspaces of~$\C[s,t]_{d}$ of
dimension~$a+1$ and~$b+1$, respectively. Suppose that
\[
 (a + b - c + 1)(d - c) \, = \, 3.
\]
Then, the cardinality of the set
\[ 
 \bigl\{ \sigma\in\pgl(2,\C)\,\colon\, \dim(V_a+V_b^\sigma) \le c+1 \bigr\} 
\]
is 
\[
  \begin{array}{lcl}
    \displaystyle \frac{1}{6} \tth ab (a^2 - 1)(b^2 - 1) & & 
     \text{if } a + b + 1 - c = 3 \text{ and } d - c = 1, \\[2ex]
    \displaystyle 6\tth\binom{a+3}{3} \binom{b+3}{3} & & 
     \text{if } a + b + 1 - c = 1 \text{ and } d - c = 3.
  \end{array}
\]
\end{theorem}

\begin{remark}
Here is a funny example: let $a=1$, $b=0$, $c=1$ and $d=4$. Then we are 
counting the number of changes of variables such that a given general binary 
quartic $G$ becomes an element of a given general linear pencil 
$\Gamma:=\langle F_0,F_1\rangle$ of binary quartics. Our formula gives the 
answer 24. However, by applying a change of variables to~$G$ we can only 
get 6 projectively different elements of~$\Gamma$. In fact, binary quartics 
have one invariant $I$ which is of degree~6 in the coefficients (see 
\cite[Section~10.2]{Dolgachev2003}). Therefore the value of $I$ coincides with 
$I(G)$ at exactly 6 elements in the pencil, obtained by solving the equation 
$I(\lambda F_0+\mu F_1)=I(G)$ for $(\lambda : \mu)\in\p^1$. The discrepancy 
between the number of changes of variables and the number of elements in 
$\Gamma$ is explained by the fact that a general quartic binary form has 4 
automorphisms.
\end{remark}

We conclude by translating the Theorem~\ref{theorem:algebraic}
into the answer to our initial problem. First of all, notice that if $V_a$ and 
$V_b$ are general and Equation~\eqref{equation:finiteness_condition} holds, then
\[
 \left\{
 \begin{array}{c}
  \sigma \in \pgl(2,\C) \text{ such that} \\
  \dim \bigl( V_a + V_b^{\sigma} \bigr) \, \leq \, c+1
 \end{array}
 \right\} = 
 \left\{
 \begin{array}{c}
  \sigma \in \pgl(2,\C) \text{ such that} \\
  \dim \bigl( V_a + V_b^{\sigma} \bigr) \, = \, c+1
 \end{array}
 \right\} .
\]
In fact, the arguments of Section~\ref{finiteness} show that the set of $\sigma 
\in \pgl(2,\C)$ such that $\dim \bigl( V_a + V_b^{\sigma} \bigr) \, \leq \, c$ is empty.
This holds because, using the notation introduced there, in this case the 
incidence variety~$\mscr{I}$ has codimension $(d-c)(a+b-c+2)$ in~$\Gr{a}{d}$, 
so its dimension is strictly smaller than the dimension of $\Gr{a}{d} \times 
\Gr{b}{d}$ and therefore the map~$\psi$ cannot be dominant.

The two subspaces $V_a$ and $V_b$ define parametrizations $f_a \colon \p^1 
\longrightarrow C_a \subseteq \p^a$ and $f_b \colon \p^1 \longrightarrow C_b 
\subseteq \p^b$. Consider the equivalence relation on triples $(\pi_a, \pi_b, 
C_c)$, where $C_c \subseteq \p^c$ is a rational curve of degree $d$ and $\pi_u 
\colon \p^c \dashrightarrow \p^u$ are linear projections such that $\pi_u (C_c) 
= C_u$ for $u \in \{a,b\}$, given by the action of~$\pgl(c+1,\C)$:
\[
 (\pi_a, \, \pi_b, \, C_c) \sim
 \bigl(\pi_a \circ \alpha, \, \pi_b \circ \alpha, \, \alpha^{-1}(C_c)\bigr) 
 \qquad \text{for every } \alpha \in \pgl(c+1,\C).
\]
We now show that there is a bijection
\[
 \left\{
 \begin{array}{c}
  \sigma \in \pgl(2,\C) \text{ such that} \\
  \dim \bigl( V_a + V_b^{\sigma} \bigr) \, = \, c+1
 \end{array}
 \right\} \longleftrightarrow
 \left\{
 \begin{array}{c}
  \text{equivalence classes under } \sim \\
  \text{of triples } (\pi_a, \pi_b, C_c)
 \end{array}
 \right\} .
\]
Starting from $\sigma\in\pgl(2,\C)$, we define $f_c \colon \p^1 \longrightarrow 
\p^c$ as the map associated to the $(c+1)$-dimensional vector space $V_a + 
V_b^{\sigma}$. Since $V_a$ and $V_b^{\sigma}$ are subspaces, we get projections 
$\pi_a$ and $\pi_b$ sending the image $C_c=f_c(\p^1)$ to $C_a$ and $C_b$, 
respectively. \\ Conversely, starting from a triple $(\pi_a, \pi_b, C_c)$ as 
above, one defines $f_c\colon\p^1\longrightarrow\p^c$ as 
$\bigl({\pi_a}_{|_{C_c}}\bigr)^{-1}\circ f_a$. Then we define $\sigma \colon 
\p^1 \longrightarrow \p^1$ as $f_b^{-1} \circ \pi_b \circ f_c$. Here we use that 
${\pi_a}_{|_{C_c}}$ and ${\pi_a}_{|_{C_c}}$ are birational because they preserve 
the degree of the curve. If we replace the triple $(\pi_a, \pi_b, C_c)$ by an 
equivalent one, then we get the same $\sigma$. One can also check that these two 
constructions are each other's inverse.
Hence it follows:

\begin{theorem}
\label{theorem:number_witnesses}
 Let $C_a \subseteq \p^a$ and $C_b \subseteq \p^b$ be 
two general rational curves of degree~$d$. Let $c$ be a natural 
number and suppose that the following holds:
\[
  (a + b + 1 - c)(d - c) \, = \, 3.
\]
Then there are, up to automorphisms of~$\p^c$, finitely many rational 
non-degenerate curves $C_c \subseteq \p^c$ of degree~$d$ together with 
linear projections 
$\pi_a \colon C_c \longrightarrow C_a$ and $\pi_b \colon 
C_c \longrightarrow C_b$.
\begin{itemize}
  \item[(1)]
    Suppose that $a + b + 1 - c = 1$ and $d - c = 3$. Then, the number of these 
curves and projections is
    \[
      \frac{1}{6} (a+3)(a+2)(a+1)(b+3)(b+2)(b+1).
    \]
  \item[(2)]
    Suppose that $a + b + 1 - c = 3$ and $d - c = 1$. Then, the number of these 
curves and projections is
    \[ 
      \frac{1}{6} \tth ab (a^2 - 1)(b^2 - 1).
    \]
\end{itemize}
\end{theorem}

\section*{Acknowledgments}

We thank Hoon Hong for many useful discussions that inspired this work, and 
Niels Lubbes for helping us with a particular case of our problem and for 
useful comments. We thank Christoph Koutschan for providing us the proof of 
a useful lemma. Matteo Gallet would like to thank Dario Portelli for teaching 
him, among many other things, important notions concerning intersection theory.

\providecommand{\bysame}{\leavevmode\hbox to3em{\hrulefill}\thinspace}
\providecommand{\MR}{\relax\ifhmode\unskip\space\fi MR }
\providecommand{\MRhref}[2]{%
  \href{http://www.ams.org/mathscinet-getitem?mr=#1}{#2}
}
\providecommand{\href}[2]{#2}

\end{document}